\documentclass[11pt]{amsart}

\usepackage[a4paper,margin=29mm]{geometry}
\usepackage{amsmath,amssymb,mathrsfs,mathtools}
\usepackage{enumitem}
\usepackage[hidelinks,hypertexnames=false]{hyperref}
\usepackage[nameinlink,capitalize]{cleveref}

\numberwithin{equation}{section}
\newcommand{\C}{\mathbb C}
\newcommand{\A}{\mathbb A}
\newcommand{\PP}{\mathbb P}
\newcommand{\OO}{\mathcal O}
\newcommand{\Bl}{\operatorname{Bl}}

\newcommand{\id}{\operatorname{id}}

\newcommand{\Span}{\operatorname{span}}

\newcommand{\Sing}{\operatorname{Sing}}
\newcommand{\rank}{\operatorname{rank}}

\theoremstyle{plain}
\newtheorem{theorem}{Theorem}[section]
\newtheorem{proposition}[theorem]{Proposition}
\newtheorem{lemma}[theorem]{Lemma}
\crefname{lemma}{Lemma}{Lemmas}
\Crefname{lemma}{Lemma}{Lemmas}
\newtheorem{corollary}[theorem]{Corollary}

\theoremstyle{definition}
\newtheorem{remark}[theorem]{Remark}
\crefname{remark}{Remark}{Remarks}
\Crefname{remark}{Remark}{Remarks}

\title{Elliptic complements of cubic hypersurfaces}
\author{Song-Yan Xie}
\address{%
  State Key Laboratory of Mathematical Sciences,
  Academy of Mathematics and Systems Science,
  Chinese Academy of Sciences, Beijing 100190, China; 
  School of Mathematical Sciences,
  University of Chinese Academy of Sciences,
  Beijing 100049, China
}
\email{xiesongyan@amss.ac.cn}
\date{\today}

\begin{document}

\begin{abstract}
Let \(D\subset\PP^n\), \(n\geqslant2\), be an arbitrary cubic hypersurface,
and let \(D_{\mathrm{red}}\) denote its reduced support. We prove that
\(\PP^n\setminus D\) is holomorphically elliptic, and hence Oka, unless
\(D_{\mathrm{red}}\) is the union of three distinct hyperplanes containing a
common codimension-two linear subspace. In the exceptional case,
\(\PP^n\setminus D\cong(\C\setminus\{0,1\})\times\C^{n-1}\), so the
complement is not Oka.

As applications, we prove that, for every elliptic curve \(E\), the space of
degree-three holomorphic maps \(E\to\PP^1\), and the space of degree-three
holomorphic self-maps of \(\PP^1\), are both holomorphically elliptic, and
hence Oka. For the second application, we identify the postcomposition
quotient with the complement of an irreducible cubic hypersurface in
\(\PP^4\).
\end{abstract}

\maketitle

\section{Introduction}\label{sec:introduction}

Oka theory grew from Oka's solution of the second Cousin problem on domains
of holomorphy~\cite{Oka1939}. Grauert extended the principle to holomorphic
bundles with homogeneous fibres over Stein spaces
\cite{GrauertApproximation,GrauertLie,GrauertFibrations}, and Gromov placed
the theory within the \(h\)-principle and introduced elliptic manifolds via
dominating holomorphic sprays~\cite[\S\S0.5--0.6]{Gromov}. Subsequent work
developed Oka principles for maps and sections with interpolation along
subvarieties, for submersions with sprays, and for subelliptic manifolds
\cite{ForstnericPrezeljExtension,ForstnericPrezelj,
ForstnericPrezeljSub,ForstnericSub,ForstnericRunge,ForstnericOka}; see
\cite{ForstnericLarusson,ForstnericToulouse,ForstnericRecent,
ForstnericBook,ForstnericICM} for broader accounts. A central practical
problem is to recognize the Oka property from concrete geometry, since its
defining approximation properties rarely decide whether a given algebraic
or quasi-projective variety is Oka. Dominating sprays, algebraic flexibility,
density properties, and blowup constructions provide some of the main
geometric mechanisms; see
\cite{ForstnericFlexibility,ArzhantsevFlexible,
KalimanKutzschebauchHypersurfaces,LarussonTruong,
KalimanKutzschebauchTruong}.

Hypersurface complements provide a particularly sharp testing ground: they lie
at the interface with logarithmic hyperbolicity, where increasing
degree is expected to produce the opposite behavior. In low degree, Hanysz
studied Oka properties of hyperplane arrangements and meromorphic-graph
complements \cite[Theorems~3.1 and~4.6]{Hanysz}, as well as affine-conic
configurations \cite[Appendix~B]{HanyszCubicMaps}. Kusakabe
proved that every irreducible singular plane-cubic complement is Oka
\cite[Corollary~4.9(2), p.~1050]{KusakabeLocalization}; see also
\cite{KusakabeComplements,ForstnericWold}.  The remaining smooth plane cubic
case was posed by Forstneri\v c--L\'arusson
\cite[Problem~B, p.~33]{ForstnericLarusson}, revisited in
\cite[\S2.3.4, pp.~758--759]{ForstnericToulouse}, and remained recorded in
\cite[Remark~5.6, p.~394]{ForstnericRecent}. Buzzard and Lu obtained
dominability by \(\C^2\) \cite[Proposition~5.1]{BuzzardLu}, but their map is
ramified, and the Oka property does not in general descend through a ramified
map.

These results leave a natural classification problem in degree three.  Since
the complement depends only on
the reduced support of \(D\), write \(D_{\mathrm{red}}\) for the hypersurface
obtained by discarding component multiplicities. Explicitly, if
\[
F=c\prod_{j=1}^{N_F}F_j^{m_j}
\]
with \(N_F\geqslant1\) and distinct irreducible homogeneous factors \(F_j\), then
\[
D_{\mathrm{red}}=\Big\{\prod_{j=1}^{N_F}F_j=0\Big\}.
\]

\begin{theorem}\label{thm:main}
Let \(D\subset\PP^n\), \(n\geqslant2\), be a cubic hypersurface; no
smoothness, irreducibility, or reducedness assumption is imposed.  Then the
following conditions are equivalent:
\begin{enumerate}[label=\textup{(\roman*)}]
\item \(\PP^n\setminus D\) is holomorphically elliptic;
\item \(\PP^n\setminus D\) is Oka;
\item \(D_{\mathrm{red}}\) is not the union of three distinct hyperplanes
containing a common linear subspace of codimension two.
\end{enumerate}
\end{theorem}

Thus there is only one obstruction among cubic supports, and it is completely
explicit. After a projective change of coordinates, the three hyperplanes have
the form
\[
Z_0=0,\qquad Z_1=0,\qquad Z_1-Z_0=0.
\]
Their complement is biholomorphic to
\[
(\C\setminus\{0,1\})\times\C^{n-1},
\]
which is not Oka. Indeed, on the chart \(Z_0\ne0\), the coordinates
\(z=Z_1/Z_0\) and \(w_j=Z_j/Z_0\), \(2\leqslant j\leqslant n\), identify the
complement with the stated product: the first coordinate avoids \(0\) and
\(1\), while the remaining coordinates are unrestricted.

The equivalence of the first two conditions follows from two general theorems
of Gromov. Elliptic manifolds are Oka \cite{Gromov}; see also
\cite[Theorem~1.1, p.~529]{ForstnericSub}. Conversely, these
complements are affine, hence Stein, and every Stein Oka manifold is elliptic
by Gromov's theorem \cite{Gromov}; see also
\cite[Proposition~5.6.15, p.~230]{ForstnericBook}. We use the
elliptic formulation because our proof constructs explicit dominating
holomorphic sprays.

\medskip
\noindent\textbf{Two applications to mapping spaces.}
\smallskip

The classification also yields two applications to natural mapping spaces in
degree three.

Let \(E\) be an elliptic curve, and let \(\mathcal R_d(E,\PP^1)\) denote the
complex manifold of degree-\(d\) holomorphic maps \(E\to\PP^1\). Bowman
proved that, for \(d=3\), its Oka property is equivalent to that of the
complement of a suitable smooth plane cubic \cite[Corollary~30]{Bowman}.
Together with \cref{thm:main}, this gives the following stronger conclusion.

\begin{theorem}
\label{thm:elliptic-functions}
For every elliptic curve \(E\), the manifold \(\mathcal R_3(E,\PP^1)\) is
holomorphically elliptic. In particular, it is Oka.
\end{theorem}

This settles the degree-three case of Forstneri\v c's Problem~7.6.23
\cite[Problem~7.6.23, p.~347]{ForstnericBook}.

\medskip

Let \(R_d\) denote the complex manifold of degree-\(d\) holomorphic self-maps
of \(\PP^1\). Equivalently, a point of \(R_d\) is represented by a pair
\([p:q]\) of binary forms of degree \(d\) without a common zero, modulo
simultaneous nonzero scaling. Hanysz proved that \(R_3\) is strongly
dominable and strongly \(\C\)-connected
\cite[Theorems~1.7 and~1.8]{HanyszCubicMaps}, while its Oka property remained
open.

\begin{theorem}
\label{thm:cubic-rational-maps}
The manifold \(R_3\) is holomorphically elliptic. In particular, it is Oka.
\end{theorem}

The connection with \cref{thm:main} comes from the postcomposition quotient of
\(R_3\), which is biholomorphic to the complement of an irreducible cubic
hypersurface in \(\PP^4\). The two applications are proved in
\cref{subsec:elliptic-functions,subsec:cubic-rational-maps}.

\medskip
\noindent\textbf{Proof strategy.}
\smallskip

We describe the geometric mechanism behind \cref{thm:main}. The smooth case
is the essential one; singular and reducible cubics are handled separately.

Let \(D=\{F=0\}\subset\PP^n\) be smooth. The cyclic cubic cover reduces
the complement to a hyperplane complement in a smooth projective cubic.
More precisely, \(\{F=1\}\) is an unramified cyclic cover of
\(\PP^n\setminus D\), and
\[
\{F=1\}=S_F\setminus B,\qquad
S_F=\{F(x)=w^3\}\subset\PP^{n+1},\qquad
B=S_F\cap\{w=0\}.
\]
It remains to construct complement-preserving sprays on \(S_F\setminus B\).

The proof has three separate tasks: to define the fibrewise flows for every
complex time, to keep their trajectories in the hyperplane complement, and
to obtain enough infinitesimal directions to dominate at every point. The
common construction handles the first two tasks and the global compatibility
of the local flows; the surface and higher-dimensional arguments establish
domination by different geometric methods.

\medskip
\noindent\textbf{Residual conic fibrations and complement-preserving sprays.}
\smallskip

A line on \(S_F\) determines a family of residual quadratic sections. Away
from planes contained in \(S_F\), these form a residual-conic fibration.
Such fibrations are classical and occur in the compact setting in
\cite[Proposition~3.7 and Remark~3.8]{KalimanZaidenberg}; the point here is to
globalize their fibrewise Hamiltonian flows while keeping every trajectory in
the hyperplane complement. The local time parameters form a line-bundle
cocycle, and identifying this cocycle allows the flows to glue. On cubic
surfaces the \(27\) conic bundles give a finite dominating family. In higher
dimensions the sprays first live on blowups along lines; a vanishing rescaling
makes them descend, after which the tangent-transfer argument of
Kaliman--Zaidenberg gives pointwise domination. We record the two cases
separately.

\begin{theorem}\label{thm:surface-source}
Let \(S\subset\PP^3\) be a smooth cubic surface, let
\(H_0\subset\PP^3\) be a hyperplane, and assume that
\(B=S\cap H_0\) is smooth.  Then the \(27\) lines on \(S\) furnish a finite
global dominating family of rank-one holomorphic sprays on
\(S\setminus B\).  In particular, \(S\setminus B\) is holomorphically
elliptic.
\end{theorem}

\begin{theorem}\label{thm:higher-source}
Let \(S\subset\PP^{n+1}\) be a smooth cubic hypersurface of dimension
\(n\geqslant3\), let \(H_0\subset\PP^{n+1}\) be a hyperplane, and put
\(B=S\cap H_0\), which need not be smooth.  Then \(S\setminus B\) admits a
finite global dominating family of holomorphic sprays on trivial line
bundles.  In particular, \(S\setminus B\) is holomorphically elliptic.
\end{theorem}

Integral singular cubics are reduced by projection from a singular point and
Kusakabe's localization theorem~\cite{KusakabeLocalization}; reducible and
nonreduced cubics reduce to affine complements of degree at most two. These
arguments are combined in \cref{sec:classification}.

\medskip
\noindent\textbf{Relation with compact cubic geometry.}
\smallskip

The preceding construction is rooted in the classical geometry of lines on
smooth projective cubic hypersurfaces; we refer to \cite{HuybrechtsCubic} for
the general theory. Smooth cubic
surfaces are uniformly rational
\cite[Example~2.4]{BogomolovBohning}, and uniformly rational varieties are
algebraically elliptic
\cite[Theorem~1.3]{ArzhantsevKalimanZaidenberg}; see also
\cite{BaneckiRetractRational} for related developments on rationality
properties.  Kaliman and Zaidenberg proved that every smooth cubic
hypersurface of dimension at least two is algebraically elliptic
\cite[Theorem~1.1]{KalimanZaidenberg}. Their construction supplies the two
geometric inputs used in our higher-dimensional argument: tangent directions
of lines through a general point span the tangent space, and the third-point
involution transports these directions to residual conics through an arbitrary
point
\cite[Lemma~3.2, Proposition~3.7 and Remark~3.8]{KalimanZaidenberg}.
These compact-cubic sprays need not preserve a hyperplane complement, since
their images may cross the deleted hyperplane. We retain the tangent-spanning
and tangent-transfer inputs, while the Hamiltonian construction above supplies
the missing boundary control.

In a subsequent development, Guo and Xie used a related line-projection
construction to prove that every smooth quartic hypersurface of dimension at
least three is Oka~\cite{GuoXieQuartic}.

The cubic theorem also lies at a natural numerical boundary for the
logarithmic canonical class. For a smooth degree-\(d\) hypersurface
\(C\subset\PP^n\), the logarithmic
canonical class is
\[
K_{\PP^n}\otimes\OO_{\PP^n}(C)\cong\OO_{\PP^n}(d-n-1).\]
For cubic hypersurfaces this bundle is trivial when \(n=2\) and anti-ample
when \(n\geqslant3\).  Such numerical behavior is related to logarithmic
hyperbolicity questions, including the Green--Griffiths--Lang philosophy,
which predicts algebraic degeneracy of entire curves in sufficiently positive
complements \cite{GreenGriffiths}. Such numerical information alone does not
imply the Oka property; the cubic result is proved by the explicit dominating
sprays constructed above.

Forstneri\v c--L\'arusson posed broader low-degree Oka questions
\cite[Question~E, pp.~33--34]{ForstnericLarusson}. Their asymptotic form leads
to the following problem.

\medskip
\noindent\textbf{Question 1.}
Fix \(d\geqslant2\).  For a general smooth hypersurface
\(H\subset\PP^n\) of degree \(d\), are \(H\) and
\(\PP^n\setminus H\) Oka when \(n\) is sufficiently large relative to
\(d\)?  What are the optimal relations between \(n\) and \(d\) in the two
cases?

\medskip
This is an opposite-direction analogue of the Kobayashi hyperbolicity
problem, in both its compact and logarithmic forms: hyperbolicity is expected
when the degree is large relative to the dimension, whereas the question
above places the dimension high relative to the degree.  The analogy is not
a formal converse, and the thresholds for the compact hypersurface and its
complement may differ.  In degree three, a smooth plane cubic is a complex
torus, while higher-dimensional smooth cubics are algebraically elliptic by
Kaliman--Zaidenberg \cite[Theorem~1.1]{KalimanZaidenberg}; hence the compact
hypersurface is Oka for every \(n\geqslant2\).  The present classification
settles the complement in the same range.  Thus the cubic case goes
beyond the asymptotic range and includes the logarithmic Calabi--Yau boundary
of a plane cubic complement.

The classification establishes holomorphic ellipticity, but the sprays produced
by the proof involve matrix exponentials and are not shown to be algebraic.
Recall that an algebraic variety is
\emph{algebraically elliptic} if it admits a dominating spray whose parameter
bundle and spray map are algebraic. This leaves a finer question.

\medskip
\noindent\textbf{Question 2.}
Let \(D\subset\PP^n\), \(n\geqslant3\), be a smooth cubic hypersurface. Is
\(\PP^n\setminus D\) algebraically elliptic?

\medskip
The plane case has a negative answer, while Question~2 remains open; see
Remark~\ref{rem:algebraic-ellipticity}.

A related open problem asks whether every smooth algebraically elliptic
\(n\)-fold admits a surjective algebraic morphism from \(\A^n\); see
\cite{ForstnericSurjective,BarthTruong}. Kusakabe proved that one additional
source dimension always suffices: every such \(n\)-fold admits a surjective
algebraic morphism from \(\A^{n+1}\)
\cite[Theorem~1.2]{KusakabeSurjective}.

\medskip
\noindent\textbf{Organization of the paper.}
\smallskip

Section~\ref{sec:sprays} recalls the Oka-theoretic tools, and
Section~\ref{sec:common} constructs the residual-conic sprays. Sections
\ref{sec:surface} and~\ref{sec:higher} prove
Theorems~\ref{thm:surface-source} and~\ref{thm:higher-source}, while Sections
\ref{sec:quadratics} and~\ref{sec:singular-projection} handle reducible and
integral singular cubics. Section~\ref{sec:classification} completes the
classification: Section~\ref{subsec:classification-proof} proves
\cref{thm:main}; the two applications appear in
\cref{subsec:elliptic-functions} and
\cref{subsec:cubic-rational-maps}.

\section{Ellipticity, sprays, and Oka principles}\label{sec:sprays}

\subsection{Sprays on complex manifolds}

Let \(M\) be a complex manifold.  A \emph{holomorphic spray on \(M\)} is a
triple \((E,p,s)\), where \(p:E\to M\) is a holomorphic vector bundle and
\(s:E\to M\) is a holomorphic map satisfying
\[
s(0_x)=x\qquad (x\in M).
\]
The vertical tangent space of \(E\) at \(0_x\) is naturally identified with
the fibre \(E_x\).  The restriction
\[
(ds)_{0_x}|_{E_x}:E_x\longrightarrow T_xM
\]
is called the \emph{vertical differential} of the spray at \(x\).  The spray
is \emph{dominating} if this map is surjective for every \(x\in M\).

A finite family of sprays \((E_j,p_j,s_j)\), \(j=1,\ldots,m\), is
\emph{dominating} if
\[
T_xM=\sum_{j=1}^{m}(ds_j)_{0_x}(E_{j,x})
\qquad (x\in M).
\]
The manifold \(M\) is \emph{elliptic} if it admits a dominating holomorphic
spray, and \emph{subelliptic} if it admits a finite dominating family. See
\cite[Definition~5.6.13, pp.~229--230]{ForstnericBook}; the elliptic notion
originates in Gromov's work \cite[\S0.5]{Gromov}, and the subelliptic notion
was introduced by Forstneri\v c \cite{ForstnericSub}.

The parameter bundles in these definitions need not be trivial.  When a
finite dominating family is defined on trivial bundles, its sprays can be
composed successively to obtain one dominating spray
\cite[Lemma~2.1]{ForstnericSub}.  We use this composition in the
higher-dimensional smooth case.  On cubic surfaces the natural rank-one
parameter bundles need not be trivial; there we first obtain subellipticity
and then use the Stein converse below.
Throughout the paper, every spray is defined globally on its full parameter
bundle. An individual spray need not dominate at every point, but the finite
families constructed below dominate jointly everywhere on their base manifold.

\subsection{Consequences used in the proof}

We record the Oka principles and localization results used below. For an
algebraic variety, the corresponding algebraic notions are
obtained by requiring the parameter bundles and spray maps to be algebraic.
Kaliman--Zaidenberg proved that algebraic subellipticity and algebraic
ellipticity are equivalent
\cite[Theorem~1.1]{KalimanZaidenbergForum}.

Gromov proved that elliptic manifolds are Oka and, conversely, that Stein Oka
manifolds are elliptic \cite{Gromov}. For projective manifolds, the converse
was recently established by Forstneri\v c and L\'arusson
\cite{ForstnericLarussonElliptic}. Forstneri\v c introduced
subellipticity and proved
\[
\text{elliptic}\Longrightarrow\text{subelliptic}
\Longrightarrow\text{Oka};
\]
see \cite[Theorem~1.1, p.~529]{ForstnericSub} and
\cite[Corollary~5.6.14 and Proposition~5.6.15,
p.~230]{ForstnericBook}. The complement of a projective hypersurface is affine
and hence Stein. Consequently, the Oka property and holomorphic ellipticity
are equivalent for the projective hypersurface complements considered here.

We shall also use the Oka principle for holomorphic fibre bundles
\cite[Theorem~5.6.5, p.~225]{ForstnericBook}
(see also~\cite[pp.~1--2]{ForstnericBookErrata}). Its special case for unramified
holomorphic coverings is
\cite[Proposition~5.6.3, p.~224]{ForstnericBook}.  If a holomorphic
fibre bundle has Oka fibre, then its total space is Oka if and only if
its base is Oka.  We apply this to line bundles, \(\C^*\)-bundles, and
finite coverings.  A holomorphic retract of an Oka manifold is Oka as
well; see \cite[Proposition~5.6.8]{ForstnericBook}.

The singular case uses two additional inputs. The first is Kusakabe's Zariski
localization theorem.

\begin{theorem}\label{thm:kusakabe-localization}
Let \(M\) be a complex manifold.  If every point of \(M\) has a
Zariski-open Oka neighborhood in \(M\), then \(M\) is Oka.  
\end{theorem}

This is \cite[Theorem~1.4]{KusakabeLocalization}. Here
``Zariski-open'' is understood in the holomorphic sense: the complement is
a closed complex subvariety.  We shall also use
two consequences from the same paper: every projective quadric
complement is Oka \cite[Corollary~4.9(1)]{KusakabeLocalization}, and
every integral singular plane-cubic complement is Oka
\cite[Corollary~4.9(2)]{KusakabeLocalization}.

The second input is Hanysz's meromorphic-graph theorem
\cite[Theorem~4.6]{Hanysz}.

Following Hanysz, define the affine graph of a holomorphic map
\(m:X\to\PP^1\) by
\[
\Gamma_m=\{(x,m(x))\in X\times\C:m(x)\ne\infty\}.
\]
Thus, if \(m(x)=\infty\), the graph has no point over \(x\), and the entire
affine fibre \(\{x\}\times\C\) lies in its complement.

\begin{theorem}\label{thm:hanysz-graph}
Let \(X\) be a complex manifold.  If
\[
m=f+\frac1g:X\longrightarrow\PP^1
\]
for holomorphic functions \(f,g\) on \(X\), then
\[
(X\times\C)\setminus\Gamma_m
\]
is Oka if and only if \(X\) is Oka.
\end{theorem}

\section{The common residual-conic construction}\label{sec:common}

Blowing up a line resolves the associated projection and produces an
affine-plane family in which the residual cubic equation is quadratic. Its
fibrewise Hamiltonian flow is parametrized by a line bundle. For cubic
surfaces the blowup is trivial; in higher dimensions the resulting spray must
later be descended to the original complement.

Let \(S\subset\PP^{n+1}\) be a smooth cubic hypersurface of dimension
\(n\geqslant2\), let \(H_0\subset\PP^{n+1}\) be a hyperplane, and put
\[
                         B=S\cap H_0,\qquad Y=S\setminus B.
\]
For the surface theorem we assume that \(B\) is smooth; for
\(n\geqslant3\), no smoothness is imposed on \(B\).  Fix a projective line
\(L\simeq\PP^1\subset S\) not contained in \(H_0\).  Such a line exists
as follows.  If \(n=2\), the smooth cubic surface \(S\) has \(27\) lines
\cite[Chapter~9]{Dolgachev}, and no one of them is contained in \(H_0\)
because \(B=S\cap H_0\) is a smooth plane cubic.  If \(n\geqslant3\),
choose \(x\in Y\) in the nonempty open line-spanning locus of
\cite[Lemma~3.2]{KalimanZaidenberg}. Since \(x\notin H_0\), none of the
lines through \(x\) is contained in \(H_0\); choose any one of the lines
supplied by that lemma. Write
\(L^\circ=L\cap Y\).

\subsection{The ambient blowup and the affine conic}
Let \(L\subset\PP^{n+1}\) be a projective line.  The rational projection
from \(L\) is not defined along \(L\), and we resolve this indeterminacy by
the blowup
\[
 b_L:Z_L=\operatorname{Bl}_L\PP^{n+1}\longrightarrow\PP^{n+1}.
\]
The exceptional divisor is
\[
E_L=b_L^{-1}(L)\simeq\PP(N_{L/\PP^{n+1}}),
\]
which records the normal directions to \(L\) and provides the missing data
needed to extend the projection.  Thus the rational projection from \(L\)
extends to a morphism
\[
p_L:Z_L\longrightarrow\PP^{n-1}.
\]
The fibre of \(p_L\) over a point of \(\PP^{n-1}\) is the plane spanned by
\(L\) and the corresponding normal direction.  In particular, \(p_L\) is a
\(\PP^2\)-bundle.

Choose homogeneous coordinates
\([z_0:z_1:y_2:\cdots:y_{n+1}]\) in which
\[
L=\{y_2=\cdots=y_{n+1}=0\}.
\]
The planes through \(L\) are parametrized
by \([u_2:\cdots:u_{n+1}]\in\PP^{n-1}\), and the blowup is the incidence
variety
\[
Z_L=
\left\{
([z_0:z_1:y],[u])\in\PP^{n+1}\times\PP^{n-1}
\;\middle|\;
y_i u_j=y_j u_i,\quad 2\leqslant i,j\leqslant n+1
\right\}.
\]
The first projection is the blowdown map \(b_L\), while the second projection
is precisely the resolved projection \(p_L\).

We use the convention that \(\PP(E)\) parametrizes one-dimensional
subspaces of the fibres of \(E\). Equivalently,
\[
Z_L\cong
\PP_{\PP^{n-1}}
\bigl(\OO^{\oplus2}\oplus\OO(-1)\bigr),
\]
where the fibre of the tautological bundle \(\OO(-1)\) at \([u]\) is the
line \(\C u\).  Thus the fibre over \([u]\) is
\(\PP(\C^2\oplus\C u)\), namely the plane spanned by \(L\) and the normal
direction \([u]\).  Removing \(D_0\) turns these projective planes into the
affine planes carrying the residual quadratic equations.

Let \(H\) denote the divisor class of
\(b_L^*\mathcal O_{\PP^{n+1}}(1)\), and let
\(\widetilde S\) and \(D_0\) denote the strict transforms of \(S\) and \(H_0\),
respectively.  Put
\[
G_L=H-E_L .
\]
The linear system of hyperplanes containing \(L\) defines the resolved
projection \(p_L\), and hence
\[
\mathcal O_{Z_L}(G_L)=p_L^*\mathcal O_{\PP^{n-1}}(1).
\]

Since \(L\not\subset H_0\), its hyperplane section has strict transform
\(D_0\) of class \(H\). Moreover, \(S\) contains \(L\) with multiplicity one:
otherwise a local equation of \(S\) would lie in \(\mathcal I_L^2\), forcing
all its first derivatives to vanish along \(L\). Hence
\begin{equation}\label{eq:divisor-identity}
[D_0]=H,\qquad
[\widetilde S]=3H-E_L=2D_0+G_L .
\end{equation}

The ambient blowup restricts to \(\widetilde S\cong\Bl_LS\). For \(n=2\),
\(L\) is a Cartier divisor on \(S\); blowing up its invertible ideal does not
change the surface, so \(\Bl_LS\cong S\). For
\(n\geqslant3\), it is nontrivial. In all
dimensions,
\[
\widetilde Y_L:=\widetilde S\setminus D_0
=\Bl_{L^\circ}Y .
\]
We write
\[
\beta_L:\widetilde Y_L\longrightarrow Y,\qquad
R_L=E_L\cap\widetilde Y_L .
\]
For \(n=2\), \(\beta_L\) is an isomorphism and \(R_L\) maps isomorphically
onto \(L^\circ\); for \(n\geqslant3\), \(R_L\) is the exceptional divisor.

Every fibre of \(p_L\) is a plane \(\Pi\) containing \(L\). If
\(\Pi\not\subset S\), the restriction of the cubic equation of \(S\) to
\(\Pi\) factors as the equation of \(L\) times a nonzero quadratic equation;
the latter defines the residual conic, possibly singular or nonreduced. If
\(\Pi\subset S\), this quadratic equation vanishes identically. Since
\(L\not\subset H_0\), the intersection \(\Pi\cap H_0\) is a line in
\(\Pi\). Hence \(D_0\) restricts to a hyperplane on every
\(\PP^2\)-fibre. Therefore,
\[
\mathcal A_L:=Z_L\setminus D_0
\]
is an affine-plane bundle over \(\PP^{n-1}\), with fibres
\[
\PP^2\setminus\PP^1\cong\C^2 .
\]
Here \(Z_L\supset\mathcal A_L\) and
\(\widetilde S\supset\widetilde Y_L\) are respectively the ambient blowup and
the transformed cubic with their boundary complements, while
\(\beta_L=b_L|_{\widetilde Y_L}\) and \(R_L=\beta_L^{-1}(L^\circ)\).

Let \(s_{\widetilde S}\) and \(s_{D_0}\) be defining sections of
\(\widetilde S\) and \(D_0\). Since
\(\widetilde S-2D_0\sim G_L\) and \(s_{D_0}\) is nowhere zero on
\(\mathcal A_L\), the quotient
\[
q_L:=\frac{s_{\widetilde S}}{s_{D_0}^{\,2}}
\in
H^0\!\left(
\mathcal A_L,
p_L^*\mathcal O_{\PP^{n-1}}(1)
\right)
\]
is well defined. Its restriction to each affine-plane fibre is a quadratic
polynomial, possibly the zero polynomial, and its zero set is precisely
\(\widetilde Y_L\). If the corresponding plane \(\Pi\) is not contained in
\(S\), this zero set is the affine part of the residual conic. If
\(\Pi\subset S\), the entire affine plane is a fibre of
\(\widetilde Y_L\to\PP^{n-1}\).

In a local affine trivialization, \(q_L\) is a quadratic polynomial in the
two fibre coordinates whose coefficients depend holomorphically on the base
parameter.

\subsection{Local coordinates and the Hamiltonian time line}

The fibrewise quadratic equations provide the directions underlying the spray
construction.  We first describe these directions locally and then identify
the line bundle that makes the corresponding Hamiltonian flows globally well
defined.

Let \((\xi_1,\xi_2)\) be affine coordinates on a fibre of
\[
p_L:\mathcal A_L\longrightarrow\PP^{n-1}.
\]
For a quadratic equation
\[
q(\xi_1,\xi_2)=0,
\]
the usual Hamiltonian vector field is
\[
V_q=
\frac{\partial q}{\partial \xi_2}
\frac{\partial}{\partial \xi_1}
-
\frac{\partial q}{\partial \xi_1}
\frac{\partial}{\partial \xi_2}.
\]
Since \(V_q(q)=0\), this vector field is tangent to every level set of
\(q\), in particular to its zero set.

The formula depends on both the affine fibre coordinates and the local
defining equation. Changing either multiplies the Hamiltonian field by a
transition factor. Thus these local fields do not define an ordinary global
vector field; instead, their time parameters must take values in a line
bundle.

Set
\[
\mathcal Q_L
=
p_L^*\OO_{\PP^{n-1}}(1)|_{\mathcal A_L},
\qquad
V_L=T_{\mathcal A_L/\PP^{n-1}} .
\]
Here \(\mathcal Q_L\) is the line bundle in which the quadratic equation
\(q_L\) takes values, while \(V_L\) is the relative tangent bundle of
\(p_L\): its fibre at \(z\in\mathcal A_L\) is
\[
(V_L)_z=\ker(dp_L)_z,
\]
the tangent space to the affine-plane fibre through \(z\). We denote their
restrictions to \(\widetilde Y_L\) by
\(\mathcal Q_L|_{\widetilde Y_L}\) and \(V_L|_{\widetilde Y_L}\).

Differentiating \(q_L\) only in the fibre directions gives the relative
differential
\[
d_{\mathrm{rel}}q_L
\in
H^0
\left(
\widetilde Y_L,
V_L^\vee\otimes\mathcal Q_L
\right).
\]
Because \(V_L\) has rank two, its determinant line bundle identifies vectors
and covectors fibrewise.  More precisely,
\[
V_L^\vee
\cong
V_L\otimes(\det V_L)^\vee ,
\]
where the identification is induced by the alternating pairing
\[
(v,w)\longmapsto v\wedge w .
\]
Hence
\[
V_L^\vee\otimes\mathcal Q_L
\cong
\operatorname{Hom}
\left(
\det V_L\otimes\mathcal Q_L^{-1},
V_L
\right).
\]

Under this identification, \(d_{\mathrm{rel}}q_L\) becomes the bundle morphism
\begin{equation}\label{eq:ham-morphism}
\mathfrak h_L:
\mathcal N_L\longrightarrow V_L|_{\widetilde Y_L},
\qquad
\mathcal N_L
:=
\left(
\det V_L\otimes\mathcal Q_L^{-1}
\right)|_{\widetilde Y_L}.
\end{equation}
We call \(\mathcal N_L\) the \emph{Hamiltonian time line}: its fibres provide
the parameters of the Hamiltonian flows. To see the construction explicitly,
choose local frames \(\omega\) of \(\det V_L\) and \(e\) of
\(\mathcal Q_L\), and write \(d_{\mathrm{rel}}q_L=\upsilon\otimes e\). The
vector \(v=\mathfrak h_L(\omega\otimes e^{-1})\) is characterized by
\(v\wedge w=\upsilon(w)\omega\) for every \(w\in V_L\). Taking \(w=v\) gives
\(\upsilon(v)=0\), so \(v\) is tangent to the zero locus of \(q_L\). In fibre
coordinates, this is the usual Hamiltonian vector
\((q_{\xi_2},-q_{\xi_1})\). Hence
\(d_{\mathrm{rel}}q_L\circ\mathfrak h_L=0\), and the image of
\(\mathfrak h_L\) is tangent to \(\widetilde Y_L\).

To identify its transition functions, we compute the class of the time line.
Since \(L\) has codimension \(n\) in \(\PP^{n+1}\), the blowup formula gives
\[
K_{Z_L}=b_L^*K_{\PP^{n+1}}+(n-1)E_L
=-(n+2)H+(n-1)E_L.
\]
Moreover,
\[
p_L^*K_{\PP^{n-1}}=-nG_L=-n(H-E_L).
\]
Subtracting these classes gives
\begin{equation}\label{eq:relative-canonical}
K_{Z_L/\PP^{n-1}}
=
-2H-E_L,
\qquad
\det T_{Z_L/\PP^{n-1}}
=
\OO_{Z_L}(2H+E_L).
\end{equation}

Since \(V_L\) is the restriction of the relative tangent bundle,
\[
\det V_L
=
\OO_{Z_L}(2H+E_L)|_{\mathcal A_L}.
\]
Together with
\[
\mathcal Q_L
=
\OO_{Z_L}(H-E_L)|_{\mathcal A_L},
\]
we obtain
\begin{equation}\label{eq:time-line-class}
\mathcal N_L
=
\OO_{Z_L}(H+2E_L)|_{\widetilde Y_L}.
\end{equation}

Since \(E_L=H-G_L\), \cref{eq:time-line-class} may be rewritten as
\[
\mathcal N_L
=\OO_{Z_L}(3H-2G_L)|_{\widetilde Y_L}.
\]
On \(\mathcal A_L=Z_L\setminus D_0\), the nowhere-zero section \(s_{D_0}\)
trivializes \(\OO_{Z_L}(H)\). Dividing by \(s_{D_0}^3\) therefore removes
the \(3H\)-term and gives
\begin{equation}\label{eq:time-line-base}
\mathcal N_L
\xrightarrow{\ \sim\ }
p_L^*\OO_{\PP^{n-1}}(-2)|_{\widetilde Y_L}.
\end{equation}
The three descriptions serve different purposes: the first defines the time
parameter intrinsically, the second records its exceptional-divisor weight,
and the third shows that its transition factors come from the base and hence
remain constant along vertical trajectories.
Thus the time line is pulled back from the base. In particular, its transition
factors are constant along every Hamiltonian trajectory, which is the
structural fact that allows the local flows to glue.

Choose \(U_\alpha\subset\PP^{n-1}\) over which \(\mathcal A_L\) and
\(\mathcal Q_L\) are trivial. A point of \(\mathcal A_L\) over \(U_\alpha\)
is written as \((u,\xi_\alpha)\), where \(u\in U_\alpha\) and
\(\xi_\alpha=(\xi_{\alpha,1},\xi_{\alpha,2})\in\C^2\) are affine fibre
coordinates. Thus \(u\) denotes a base variable, \(\xi_\alpha\) the two
fibre variables, and \(\tau_\alpha\) below the Hamiltonian time coordinate.

These affine coordinates are not canonical: two local identifications of
the same fibre with \(\C^2\) may choose different origins. Hence the
transition map on an overlap is generally affine rather than linear.

In a nowhere-zero frame \(e_\alpha\) of \(\mathcal Q_L\), write
\(q_L=q_\alpha e_\alpha\). Then \(q_\alpha\) is quadratic in the fibre
variables and
\begin{equation}\label{eq:local-conic}
\widetilde Y_L|_{U_\alpha}
=
\{
q_\alpha(u,\xi_{\alpha,1},\xi_{\alpha,2})=0
\}.
\end{equation}

A point \(z\in\widetilde Y_L\) is \emph{relative regular} if
\(d_{\mathrm{rel}}q_L(z)\neq0\). At such a point, the scheme-theoretic zero
set of \(q_L\) in the affine-plane fibre is a smooth curve near \(z\). Points
on a contained-plane fibre are not relative regular, even though that fibre
is a smooth surface.

On overlaps \(U_\alpha\cap U_\beta\), the coordinates and equations satisfy
\begin{align}
\xi_\beta
&=
A_{\beta\alpha}(u)\xi_\alpha+b_{\beta\alpha}(u),
\\
q_\beta(u,\xi_\beta)
&=
\lambda_{\beta\alpha}(u)
q_\alpha(u,\xi_\alpha),
\end{align}
where \(A_{\beta\alpha}(u)\) is an invertible holomorphic \(2\times2\)
matrix, while \(b_{\beta\alpha}(u)\) records the change of origin, and
\(\lambda_{\beta\alpha}(u)\) is nowhere zero. In particular, the vertical
differential of the coordinate change is \(A_{\beta\alpha}(u)\); the
translation term has zero vertical derivative.

\subsection{The cocycle and gluing of the Hamiltonian flows}

The transition factor has two sources: the change of the line-bundle-valued
equation \(q_L\) and the change of the fibrewise area form. Their quotient is
the cocycle of \(\mathcal N_L\).

In the local coordinates of the previous subsection, the vertical
Hamiltonian vector field associated with
\[
q_\alpha(u,\xi_{\alpha,1},\xi_{\alpha,2})
\]
is
\begin{equation}\label{eq:local-hamiltonian}
V_\alpha
=
(q_\alpha)_{\xi_{\alpha,2}}
\frac{\partial}{\partial\xi_{\alpha,1}}
-
(q_\alpha)_{\xi_{\alpha,1}}
\frac{\partial}{\partial\xi_{\alpha,2}} .
\end{equation}
Let
\[
\phi_{\beta\alpha}
(u,\xi_\alpha)
=
(u,A_{\beta\alpha}(u)\xi_\alpha+b_{\beta\alpha}(u))
\]
be the affine transition map between two local trivializations.  The
corresponding local equations satisfy
\[
q_\beta(u,\xi_\beta)
=
\lambda_{\beta\alpha}(u)
q_\alpha(u,\xi_\alpha).
\]
We will show that
\begin{equation}\label{eq:field-transition}
V_\beta
=
\frac{\lambda_{\beta\alpha}}
{\det A_{\beta\alpha}}
(\phi_{\beta\alpha})_*V_\alpha .
\end{equation}

In Formula~\eqref{eq:field-transition}, \(\lambda_{\beta\alpha}\) records the
change of the local defining equation, while \(\det A_{\beta\alpha}\) records
the change of the fibrewise area form. Their quotient
\[
\frac{\lambda_{\beta\alpha}}{\det A_{\beta\alpha}}
\]
is the transition factor of
\(\mathcal N_L=\det V_L\otimes\mathcal Q_L^{-1}\); its reciprocal governs the
Hamiltonian time parameter.

To verify the vector-field transformation formula, fix a base point
\(u\), and abbreviate
\[
A=A_{\beta\alpha}(u),\qquad
b=b_{\beta\alpha}(u),\qquad
\lambda=\lambda_{\beta\alpha}(u).
\]
We regard the fibre gradient
\[
\nabla_{\xi_\alpha}q_\alpha
\]
as a column vector.  Differentiating
\[
q_\beta(u,A\xi_\alpha+b)
=
\lambda q_\alpha(u,\xi_\alpha)
\]
with respect to \(\xi_\alpha\), the chain rule gives
\[
A^T
\nabla_{\xi_\beta}q_\beta(u,A\xi_\alpha+b)
=
\lambda
\nabla_{\xi_\alpha}q_\alpha(u,\xi_\alpha).
\]
Hence
\begin{equation}\label{eq:gradient-transition}
\nabla_{\xi_\beta}q_\beta(u,A\xi_\alpha+b)
=
\lambda A^{-T}
\nabla_{\xi_\alpha}q_\alpha(u,\xi_\alpha),
\end{equation}
where
\[
A^{-T}=(A^{-1})^T
\]
is the inverse transpose of \(A\).  It appears because the gradient is a
covector, whereas tangent vectors transform by \(A\).

Let
\[
J=
\begin{pmatrix}
0&1\\
-1&0
\end{pmatrix}.
\]
The Hamiltonian vector field can be written in matrix form as
\[
V_\alpha
=
J\nabla_{\xi_\alpha}q_\alpha .
\]
For
\[
A=
\begin{pmatrix}
a_{11}&a_{12}\\
a_{21}&a_{22}
\end{pmatrix},
\qquad
\Delta=\det A,
\]
one has the elementary identity
\[
JA^{-T}
=
\Delta^{-1}
\begin{pmatrix}
-a_{12}&a_{11}\\
-a_{22}&a_{21}
\end{pmatrix}
=
\Delta^{-1}AJ .
\]
Combining this identity with
\cref{eq:gradient-transition}, and observing that the translation term
\(b\) does not contribute to the vertical differential, we obtain
\[
V_\beta
(\phi_{\beta\alpha}(u,\xi_\alpha))
=
\frac{\lambda}{\det A}
A V_\alpha(u,\xi_\alpha),
\]
which is exactly \cref{eq:field-transition}.

Set \(\kappa_{\beta\alpha}=\lambda_{\beta\alpha}/\det A_{\beta\alpha}\),
and denote the local flows by \(\Phi_\alpha^\tau\) and
\(\Phi_\beta^\tau\). Since
\(\kappa_{\beta\alpha}\) depends only on the base parameter \(u\), it is constant
along every vertical Hamiltonian trajectory.  Therefore
\[
\Phi_\beta^\tau\circ\phi_{\beta\alpha}
=
\phi_{\beta\alpha}
\circ
\Phi_\alpha^{\kappa_{\beta\alpha}\tau}.
\]
Equivalently, the local time coordinates satisfy
\begin{equation}\label{eq:time-transition}
\tau_\beta
=
\frac{\det A_{\beta\alpha}}
{\lambda_{\beta\alpha}}
\tau_\alpha .
\end{equation}
If local frames \(\nu_\alpha\) of the time line satisfy
\[
\nu_\beta=
\frac{\lambda_{\beta\alpha}}{\det A_{\beta\alpha}}\nu_\alpha,
\]
then the same vector in a fibre is written
\(\tau_\alpha\nu_\alpha=\tau_\beta\nu_\beta\), which gives
\cref{eq:time-transition}. The same identity ensures that the infinitesimal
motions \(\tau_\beta V_\beta\) and
\(\tau_\alpha(\phi_{\beta\alpha})_*V_\alpha\) agree on overlaps. This is the
transition law of
\(\mathcal N_L=\det V_L\otimes\mathcal Q_L^{-1}\). The factors
\(\det A_{\beta\alpha}/\lambda_{\beta\alpha}\) satisfy the cocycle condition
on triple overlaps and depend only on the base parameter. They are therefore
constant along Hamiltonian trajectories, so the local flows glue on the total
space of \(\mathcal N_L\).

\subsection{The entire flow}

The Hamiltonian fields are affine-linear along the affine-plane fibres. Since
the relative equation \(q_\alpha\) is quadratic in the fibre coordinates, the
local Hamiltonian field has the form
\[
\dot{\xi}=M_\alpha(u)\xi+c_\alpha(u),
\]
where \(M_\alpha(u)\) is a \(2\times2\) matrix and
\(c_\alpha(u)\) is a vector depending holomorphically on the base parameter
\(u\).

Affine-linear differential equations admit entire solutions in the complex
time parameter.  Explicitly, for \(\tau\in\C\), the local flow is
\begin{equation}\label{eq:matrix-flow}
\Phi_\alpha^\tau(u,\xi)
=
\left(
u,
e^{\tau M_\alpha(u)}\xi
+
\int_0^\tau
       e^{(\tau-s)M_\alpha(u)}
       c_\alpha(u)\,ds
\right).
\end{equation}
Since the integrand is entire in the integration variable, the integral is path independent and
depends holomorphically on all variables. Hence \cref{eq:matrix-flow} defines
the flow for every \(\tau\in\C\). By the
Hamiltonian identity
\[
V_\alpha(q_\alpha)=0,
\]
the flow preserves the level sets of \(q_\alpha\).  Hence each trajectory
remains in the same fibre of \(p_L|_{\widetilde Y_L}\).

The construction takes place on
\(\widetilde Y_L=\widetilde S\setminus D_0\); no extension of the spray
across \(D_0\) is needed.

By \cref{eq:field-transition,eq:time-transition}, these entire local flows
glue to a global holomorphic spray
\begin{equation}\label{eq:upstairs-spray}
\Phi_L:\mathcal N_L\longrightarrow\widetilde Y_L,
\qquad
\Phi_L(0_z)=z .
\end{equation}

Formula~\eqref{eq:matrix-flow} involves only the matrix exponential and an
integral; it never divides by the discriminant of the quadratic equation.
Hence the spray remains holomorphic when a residual conic degenerates. For
example, the smooth parabola \(q=y-x^2\) gives
\[
V=
\frac{\partial}{\partial x}
+
2x\frac{\partial}{\partial y},
\qquad
\Phi^\tau(x,y)
=
(x+\tau,y+2\tau x+\tau^2).
\]
The nodal conic \(q=xy\) gives the complete field
\(V=x\partial_x-y\partial_y\), which vanishes at the node. If \(q_L\)
vanishes identically on an affine-plane fibre, then so does its Hamiltonian
field, and the spray is the identity on that fibre. The same matrix formula
therefore remains valid for singular or nonreduced conics and for the
identically zero quadratic.

\begin{proposition}\label{prop:common-spray}
For every \(n\geqslant2\) and every line \(L\subset S\) not contained in
\(H_0\), the map
\[
\Phi_L:\mathcal N_L\longrightarrow\widetilde Y_L
\]
is a holomorphic spray defined on all of \(\mathcal N_L\). It is vertical for
\(p_L\), preserves every fibre of \(p_L|_{\widetilde Y_L}\), and at every
relative regular point \(z\) satisfies
\[
(d\Phi_L)_{0_z}(\mathcal N_{L,z})
=
\ker d(p_L|_{\widetilde Y_L})_z .
\]
\end{proposition}

\begin{proof}
The local flows glue by
\cref{eq:field-transition,eq:time-transition}, and
\cref{eq:matrix-flow} shows that the resulting spray is entire in the time
parameter, including over singular fibres. Since it is vertical and preserves
\(q_L\), it preserves every fibre of \(p_L|_{\widetilde Y_L}\). Its vertical
differential at the zero section is the Hamiltonian morphism \(\mathfrak h_L\).
At a relative regular point, \(p_L|_{\widetilde Y_L}\) is a submersion whose
kernel is the one-dimensional tangent space to the smooth fibre curve; the
nonzero vector \(\mathfrak h_L\) spans this kernel.
\end{proof}

\section{The surface case: the 27 lines}
\label{sec:surface}

We first specialize the common construction to the borderline case \(n=2\).
Thus \(S\subset\PP^3\) is a smooth cubic surface and
\(B=S\cap H_0\) is a smooth plane cubic. The \(27\) lines on \(S\) will
produce conic-bundle sprays whose directions jointly span \(TY\). Here
\(B\sim-K_S\), so \(K_S+B=0\); this is the logarithmic Calabi--Yau boundary
case.

Since a smooth plane cubic is irreducible, the hyperplane \(H_0\) contains
none of the \(27\) lines on \(S\). Since \(B\sim H\) and \(H\cdot L=1\), the
scheme-theoretic intersection \(B\cap L\) has length one and is therefore a
single reduced point. In particular, the affine part
\[
L^\circ:=L\setminus B
\]
is isomorphic to \(\C\).

\subsection{Conic bundles associated with lines}

Let \(L\subset S\) be one of the \(27\) lines. Choose linear forms
\(\ell_0,\ell_1\) with \(L=\{\ell_0=\ell_1=0\}\), and write
\[
F_S=\ell_0\Theta_0+\ell_1\Theta_1,
\]
for quadratic forms \(\Theta_0,\Theta_1\). These forms have no common zero on \(L\),
since such a point would be singular on \(S\). Hence projection from \(L\)
extends to the conic-bundle morphism
\begin{equation}\label{eq:surface-conic-map}
\begin{gathered}
\pi_L:S\longrightarrow\PP^1,\\
\pi_L(a)=[\ell_0(a):\ell_1(a)]\quad(a\notin L),\\
\pi_L(a)=[-\Theta_1(a):\Theta_0(a)]\quad(a\in L).
\end{gathered}
\end{equation}

On \(S\), the equation \(\ell_0\Theta_0+\ell_1\Theta_1=0\) shows that the limiting
ratio \([\ell_0:\ell_1]\) along \(L\) is \([-\Theta_1:\Theta_0]\). Since \(\Theta_0\) and \(\Theta_1\) have no
common zero on \(L\), this formula is defined everywhere and proves that the
projection extends across \(L\).

A plane containing \(L\) cuts the cubic surface into the fixed line \(L\) and
a residual conic.  Hence the fibres of \(\pi_L\) have divisor class
\[
H-L .
\]
This is the classical conic bundle structure associated with a line on a
smooth cubic surface.

The boundary curve \(B\) is a bisection because
\[
B\cdot(H-L)
=
H\cdot(H-L)
=
3-1=2.
\]
Indeed, \(B\) contains no fibre component, so the restriction of \(\pi_L\) to
\(B\) has degree equal to this intersection number. Thus every residual conic
meets \(B\) in a scheme of length two; after removing this intersection, it
becomes the affine conic carrying the spray of \cref{prop:common-spray}.

\subsection{Two skew conic bundles dominate}

Choose six pairwise skew lines on \(S\) and contract them. The resulting
blowdown \(S\to\PP^2\) identifies \(S\) with the blowup of \(\PP^2\) at six
points in general position. We use the standard divisor-class description of
the \(27\) lines and their incidence relations; see
\cite[Chapter~9]{Dolgachev}.

Let
\(\ell\) be the pullback of the class of a line in \(\PP^2\), and let
\(e_1,\ldots,e_6\) be the exceptional curves.  Then
\[
\ell^2=1,\qquad
\ell\cdot e_i=0,\qquad
e_i\cdot e_j=-\delta_{ij},
\]
and
\[
H=-K_S=3\ell-\sum_{i=1}^6e_i .
\]

The \(27\) lines on \(S\) have divisor classes
\begin{equation}\label{eq:line-classes}
e_i,\qquad
L_{ij}=\ell-e_i-e_j,\qquad
Q_i=2\ell-\sum_{k\ne i}e_k .
\end{equation}
Every line meets ten other lines and is skew to sixteen, while two skew
lines have exactly five common transversals.

Let \(L,M\subset S\) be skew lines.  The two conic bundles define the morphism
\[
\Psi_{L,M}:=(\pi_L,\pi_M):
S\longrightarrow\PP^1\times\PP^1 .
\]
The fibre classes are \(H-L\) and \(H-M\), and therefore
\begin{equation}\label{eq:fibre-intersection}
(H-L)\cdot(H-M)
=
H^2-H\cdot L-H\cdot M+L\cdot M
=
1 .
\end{equation}
The two fibre classes are effective, and their intersection number is one;
hence fibres from the two pencils cannot be disjoint. Thus every fibre of
\(\pi_L\) meets every fibre of \(\pi_M\), so
\(\Psi_{L,M}\) is surjective. Since general fibres meet in one point, it is
birational.

To locate the failure of local invertibility, let \(\Gamma\) be a curve
contracted by the birational morphism \(\Psi_{L,M}\). The intersection form is
negative definite on the exceptional locus of a birational morphism of smooth
surfaces, so \(\Gamma^2<0\). Since \(K_S=-H\), adjunction gives
\[
\Gamma^2=2p_a(\Gamma)-2+H\cdot\Gamma ,
\]
and \(H\cdot\Gamma\geqslant1\) since \(H\) is ample. Hence
\(p_a(\Gamma)=0\) and \(H\cdot\Gamma=1\), so \(\Gamma\) is a line in
the cubic embedding.

A line \(N\neq L,M\) on \(S\) is contracted by
\(\Psi_{L,M}\) precisely when it is contained in fibres of both conic bundle
structures.  Since the fibre classes of \(\pi_L\) and \(\pi_M\) are
\(H-L\) and \(H-M\), respectively, this is equivalent to
\[
(H-L)\cdot N=(H-M)\cdot N=0 .
\]
For every line \(N\subset S\), one has \(H\cdot N=1\).  Hence
\[
(H-L)\cdot N=1-L\cdot N,
\qquad
(H-M)\cdot N=1-M\cdot N .
\]
Therefore the above condition is equivalent to
\[
L\cdot N=M\cdot N=1,
\]
which means precisely that \(N\) intersects both \(L\) and \(M\).

The two distinguished lines are not contracted: since \(L\cdot M=0\),
\[
(H-M)\cdot L=(H-L)\cdot M=1.
\]

Consequently, the exceptional locus of \(\Psi_{L,M}\) is the union
of the lines meeting both \(L\) and \(M\):
\begin{equation}\label{eq:surface-exceptional-locus}
\operatorname{Exc}(\Psi_{L,M})
=
\bigcup_{\substack{N\subset S\ \mathrm{a\ line}\\
N\cdot L=N\cdot M=1}}N .
\end{equation}

\begin{lemma}\label{lem:skew-pair}
For every point \(a\in S\), there exist skew lines
\[
L,M\subset S
\]
such that
\[
\Psi_{L,M}:S\longrightarrow\PP^1\times\PP^1
\]
is a local biholomorphism at \(a\).
\end{lemma}

\begin{proof}
By \eqref{eq:surface-exceptional-locus}, it suffices to choose skew lines
\(L,M\) such that no line through \(a\) meets both of them.

A point of a smooth cubic surface lies on at most three lines, since all such
lines are contained in the tangent-plane section, which is a plane cubic.
We consider the possible configurations.

If \(a\) lies on no line, any skew pair works.  If \(a\) lies on a unique line
\(A\), choose \(L=A\) and any line \(M\) skew to \(A\).  Then no line through
\(a\) can meet both \(L\) and \(M\).

Suppose that \(a\) lies on exactly two lines \(A_1,A_2\). Among the sixteen
lines skew to \(A_1\), at most nine meet \(A_2\), since \(A_2\) meets ten lines
in total and one of them is \(A_1\). We may therefore choose a line \(M\) skew
to both \(A_1\) and \(A_2\). With \(L=A_1\), neither line through \(a\) meets
both \(L\) and \(M\).

It remains to consider the Eckardt case, where three lines
\(A_1,A_2,A_3\) pass through \(a\).  The tangent-plane section gives
\[
A_1+A_2+A_3\sim H .
\]
Every other line meets exactly one of the \(A_i\), and each \(A_i\) meets
eight such external lines, meaning lines not among \(A_1,A_2,A_3\). Choose an
external line \(L\) meeting \(A_1\).
The skew pair \(L,A_2\) has five common transversals: \(A_1\) and four of the
eight external lines meeting \(A_2\). The other four external lines meeting
\(A_2\) are therefore disjoint from \(L\); choose one of them as \(M\). Thus
\[
L\cdot A_1=M\cdot A_2=1,
\]
while all other intersections between
\(\{L,M\}\) and \(\{A_1,A_2,A_3\}\) vanish.  Thus none of the three lines
through \(a\) meets both \(L\) and \(M\).

In all cases we have chosen skew lines \(L,M\) such that
\(a\notin\operatorname{Exc}(\Psi_{L,M})\).  Since
\(\Psi_{L,M}\) is a birational morphism between smooth surfaces and is an
isomorphism away from its exceptional locus, it is a local biholomorphism at
\(a\).
\end{proof}

\subsection{The finite family of surface sprays}

For every line \(L\subset S\), identify \(\widetilde Y_L\) with \(Y\) via
\(\beta_L\), and denote the transported spray of
\cref{prop:common-spray} by
\begin{equation}\label{eq:surface-line-spray}
s_L:\mathcal N_L\longrightarrow Y .
\end{equation}
Under this identification, \(p_L=\pi_L\).

\begin{proposition}\label{prop:surface-line-spray}
For every line \(L\subset S\), the map \(s_L\) is a rank-one holomorphic
spray. At every point \(a\in Y\) where \(\pi_L\) is a submersion,
\[
(ds_L)_{0_a}(\mathcal N_{L,a})
=
\ker(d\pi_L)_a .
\]
Thus its infinitesimal direction is tangent to the residual affine conic.
\end{proposition}

\begin{proof}
This follows directly from \cref{prop:common-spray}, since
\(p_L=\pi_L\) under the identification \(\widetilde Y_L=Y\).
\end{proof}

\begin{proof}[Proof of \cref{thm:surface-source}]
Let \(s_L:\mathcal N_L\to Y\) be the spray associated with each of the
\(27\) lines. For \(a\in Y\), \cref{lem:skew-pair} gives skew lines \(L,M\)
such that \((\pi_L,\pi_M)\) is a local biholomorphism at \(a\). Therefore
\[
\ker(d\pi_L)_a+\ker(d\pi_M)_a=T_aY.
\]
By \cref{prop:surface-line-spray}, these two lines are the infinitesimal
directions of \(s_L\) and \(s_M\). The \(27\) sprays consequently form a
finite dominating family. Thus \(Y\) is subelliptic and Oka; being affine and
Stein, it is elliptic by \cref{sec:sprays}.
\end{proof}

\section{The higher-dimensional case}
\label{sec:higher}

In higher dimensions the common spray lives on a nontrivial blowup. For a line
\(L\subset S\) not contained in \(H_0\), \cref{sec:common} gives a spray
\[
\Phi_L:\mathcal N_L\longrightarrow\widetilde Y_L .
\]
It need not be constant on the exceptional fibres of
\(\beta_L:\widetilde Y_L\to Y\), so it does not descend directly. Rescaling
the time parameter by a section vanishing on the exceptional divisor makes
the modified spray fix that divisor pointwise and descend to \(Y\).
Kaliman--Zaidenberg's tangent-transfer argument gives pointwise domination,
and Noetherianity reduces the resulting family to a finite one.

\subsection{Exceptional scaling and proper-fibre descent}

Let \(R_L=E_L\cap\widetilde Y_L\) be the exceptional divisor. By
\cref{eq:time-line-class},
\[
\mathcal N_L
=
\mathcal O_{Z_L}(H+2E_L)|_{\widetilde Y_L}.
\]
Hence the defining sections of \(D_0\) and \(E_L\) give
\begin{equation}\label{eq:eta}
\eta_L:=s_{D_0}s_{E_L}^{2}|_{\widetilde Y_L}
\in H^0(\widetilde Y_L,\mathcal N_L).
\end{equation}
The exponent two is dictated by the \(2E_L\)-term in the class of
\(\mathcal N_L\); it is not an extra vanishing requirement for descent.
In particular, \(\eta_L|_{R_L}=0\). The modified spray
\begin{equation}\label{eq:modified-upstairs-spray}
\widehat s_L:\widetilde Y_L\times\C\longrightarrow\widetilde Y_L,
\qquad
\widehat s_L(z,t)=
\Phi_L(z,t\eta_L(z)).
\end{equation}
Here \(t\) is the coordinate on the trivial parameter line; multiplication
by \(\eta_L(z)\) turns it into the Hamiltonian time
\(t\eta_L(z)\in\mathcal N_{L,z}\). Moreover, \(\eta_L\) vanishes on \(R_L\); hence
\(\widehat s_L\) fixes \(R_L\) pointwise and has the fibrewise constancy
required for descent.

\begin{lemma}
\label{lem:exceptional-constancy}

Let \(\beta:\widehat X\to X\) be the blowup of a smooth centre, with
exceptional divisor \(E\), and let \(\Phi:\mathcal N\to\widehat X\) be a
holomorphic spray. If \(\eta\in H^0(\widehat X,\mathcal N)\) vanishes along
\(E\), then
\[
\widehat s(z,t)=\Phi(z,t\eta(z))
\]
fixes \(E\) pointwise, and \(\beta\circ\widehat s\) is constant on every
fibre of \(\beta\times\operatorname{id}_{\C}\).

\end{lemma}

\begin{proof}
For \(z\in E\), the equality \(\eta(z)=0_z\) gives
\(\widehat s(z,t)=\Phi(z,0_z)=z\). Thus \(E\) is fixed pointwise.
Every fibre of \(\beta\) is either a singleton or is contained in \(E\);
hence \(\beta\circ\widehat s\) is constant on the fibres of
\(\beta\times\operatorname{id}_{\C}\).
\end{proof}

We apply this observation to \(\beta_L:\widetilde Y_L\to Y\) through the
following analytic descent lemma.

\begin{lemma}
\label{lem:analytic-descent}

Let \(f:\widehat X\to X\) be a proper bimeromorphic holomorphic map between
complex manifolds, with \(X\) normal, and let \(T\) be a complex manifold.
If a holomorphic map \(G:\widehat X\times T\to A\), where
\(A\subset\C^N\) is a closed complex subvariety, is constant on the fibres
of \(f\times\operatorname{id}_T\), then it factors uniquely through a
holomorphic map \(g:X\times T\to A\).

\end{lemma}

\begin{proof}
The product \(X\times T\) is normal, and
\(f\times\operatorname{id}_T\) is a proper bimeromorphic modification.
Hence
\[
(f\times\operatorname{id}_T)_*
\mathcal O_{\widehat X\times T}
=
\mathcal O_{X\times T};
\]
by the analytic Riemann extension theorem
\cite[Chapter~8, \S1]{GrauertRemmert}. The coordinate functions of \(G\)
therefore descend holomorphically to \(X\times T\). They define a map to
\(\C^N\), whose image lies in \(A\) because the defining equations of \(A\)
vanish after pullback. Uniqueness follows from surjectivity.
\end{proof}

The variety \(Y=S\cap(\PP^{n+1}\setminus H_0)\) is closed in the affine chart
\(\PP^{n+1}\setminus H_0\cong\C^{n+1}\), and is therefore affine. Choose an
algebraic closed embedding \(\varepsilon:Y\hookrightarrow\C^N\). By
\cref{lem:exceptional-constancy}, the map
\(\varepsilon\circ\beta_L\circ\widehat s_L\) is constant on the fibres of
\(\beta_L\times\id_\C\). Hence \cref{lem:analytic-descent} yields the
holomorphic spray
\begin{equation}\label{eq:downstairs-spray}
s_L:Y\times\C\longrightarrow Y,
\qquad
s_L(y,0)=y .
\end{equation}

\subsection{Algebraicity at zero time}

The descended spray is holomorphic. To make its tangent-spanning locus
Zariski open and later extract finitely many sprays, we also need its
zero-time field to be algebraic. This stronger infinitesimal statement holds
because the Hamiltonian direction is algebraic, even though the complete flow
need not be.

Recall from \cref{eq:ham-morphism} the Hamiltonian bundle morphism
\(\mathfrak h_L:\mathcal N_L\to V_L|_{\widetilde Y_L}\), and define
\begin{equation}\label{eq:zero-time-field}
W_L(y)=
\left.\frac{\partial}{\partial t}\right|_{t=0}s_L(y,t).
\end{equation}

On \(Y\setminus L^\circ\), the map
\(\beta_L:\widetilde Y_L\to Y\) is an isomorphism. Hence, for
\(z=\beta_L^{-1}(y)\),
\begin{equation}\label{eq:differentiated-descent}
W_L(y)
=
d(\beta_L)_z
\bigl(\mathfrak h_L(\eta_L(z))\bigr).
\end{equation}
The right-hand side is algebraic, since it is obtained from the algebraic
section \(\eta_L\), the algebraic Hamiltonian morphism
\(\mathfrak h_L\), and the algebraic blowup map \(\beta_L\).

At every relative regular point \(z\notin R_L\) of the residual affine conic,
\(\mathfrak h_L(\eta_L(z))\) is nonzero and spans the tangent line of the
conic.  Therefore \(W_L\) is nonzero and spans the corresponding tangent line
at every regular point of the residual affine conic in
\(Y\setminus L^\circ\).

It remains to extend this algebraic vector field across \(L^\circ\).  Since
\[
Y=S\setminus B
\]
is smooth, its tangent sheaf \(T_Y\) is locally free and hence reflexive.
Moreover,
\[
\operatorname{codim}_Y L^\circ=n-1\geqslant2 .
\]
For the inclusion
\[
j:Y\setminus L^\circ\hookrightarrow Y,
\]
reflexive extension gives
\[
T_Y
\cong
j_*(T_Y|_{Y\setminus L^\circ});
\]
see \cite[Proposition~1.6]{HartshorneReflexive}.  Hence the algebraic vector
field on \(Y\setminus L^\circ\) extends uniquely to an algebraic section of
\(T_Y\) on all of \(Y\). This is the algebraic Hartogs principle for a
reflexive sheaf across a subset of codimension at least two.

The algebraic extension and \(W_L\) agree on the dense open set
\(Y\setminus L^\circ\), hence everywhere. Since the descended spray fixes
\(L^\circ\), the field \(W_L\) vanishes there. Thus \(W_L\) is algebraic on
all of \(Y\), as required for finite extraction.

\begin{proposition}\label{prop:higher-one-line}
For every line \(L\subset S\) not contained in \(H_0\), there exists a global
holomorphic spray \(s_L:Y\times\C\to Y\) whose zero-time field \(W_L\)
is algebraic and vanishes along \(L^\circ\). At every
\(y\in Y\setminus L^\circ\) where the residual affine conic is regular,
\(W_L(y)\) is nonzero and spans its tangent line.
\end{proposition}

\subsection{Tangent transfer and pointwise spanning}

Pointwise domination follows from the tangent-transfer mechanism of
Kaliman--Zaidenberg, which transports tangent directions of lines through a
general point to residual conics through an arbitrary point.

We use \cite[Lemma~3.2]{KalimanZaidenberg} in the following form. There exists
a nonempty Zariski open subset \(U\subset S\) such that, for every \(x\in U\), one can
choose lines
\[
L_1,\ldots,L_n\subset S
\]
through \(x\) satisfying
\begin{equation}\label{eq:kz-spanning-at-x}
T_xS=
\operatorname{Span}(T_xL_1,\ldots,T_xL_n).
\end{equation}

The transfer is provided by the third-point involution used by
Kaliman--Zaidenberg
\cite[Corollary~3.6(c), Proposition~3.7]{KalimanZaidenberg}. In introducing
this birational map, they refer to \cite[Example~2.4]{BogomolovBohning} and
\cite{BogomolovKarzhemanovKuyumzhiyan}. For fixed
\(\rho\in S\), this rational involution \(\iota_\rho\) sends a general point \(p\in S\)
to the residual third point of the line \(\langle \rho,p\rangle\cap S\). Suppose
that the secant \(\ell=\langle x,y\rangle\) meets \(S\) in the reduced divisor
\(x+\rho+y\). For completeness, the local property needed below follows
directly from projection. Let \(\pi_\rho\) denote projection from \(\rho\). Since
the intersection \(S\cdot\ell=x+\rho+y\) is reduced, the line \(\ell\) is
tangent to \(S\) at neither \(x\) nor \(y\). Hence \(d\pi_\rho\) is an
isomorphism at both points. After shrinking neighborhoods, the holomorphic
inverse function theorem identifies neighborhoods of \(x\) and \(y\) with a
common neighborhood of the direction of \(\ell\). The transition between the
two inverse branches is the third-point involution \(\iota_\rho\). Thus
\(\iota_\rho\) is locally biholomorphic at \(x\) and \(y\). If
\[
\Pi_i=\langle L_i,y\rangle,\qquad
\Pi_i\cap S=L_i+\Gamma_i,
\]
then \(\iota_\rho\) maps the germ of \(L_i\) at \(x\) to the germ of
\(\Gamma_i\) at \(y\): for \(p\in L_i\) near \(x\), the third point of
\(S\cap\langle \rho,p\rangle\) is \(\iota_\rho(p)\in\Gamma_i\). Therefore
\begin{equation}\label{eq:kz-tangent-transfer}
T_y\Gamma_i=d\iota_\rho(T_xL_i)
\end{equation}
holds.

\begin{lemma}\label{lem:pointwise-lines}
For every \(y\in Y\), there exist lines
\[
L_1,\ldots,L_n\subset S,
\]
none contained in \(H_0\) and none passing through \(y\), such that the
residual conics through \(y\) are regular at \(y\) and
\[
T_yY=
\operatorname{Span}
(W_{L_1}(y),\ldots,W_{L_n}(y)).
\]
\end{lemma}

\begin{proof}
Fix \(y=[a_0:\cdots:a_{n+1}]\in Y\), and let \(F=0\) be an equation
of \(S\). Write \(\mathbb T_yS\) for the projective tangent hyperplane to
\(S\) at \(y\). For \(x\ne y\), the line \(\langle x,y\rangle\) is tangent
to \(S\) at \(y\) precisely when \(x\in\mathbb T_yS\). It is tangent at
\(x\) precisely when \(\sum_{j=0}^{n+1}a_jF_{x_j}(x)=0\). Thus the two bad
loci are \(S\cap\mathbb T_yS\) and
\[
P_y:=
S\cap
\left\{
\sum_{j=0}^{n+1}a_j
\frac{\partial F}{\partial x_j}=0
\right\}.
\]
Both loci are proper. The smooth cubic \(S\) cannot be contained in its tangent
hyperplane \(\mathbb T_yS\), so the first is proper. If \(P_y=S\), then the
quadratic polar \(\sum a_jF_{x_j}\) would vanish on \(S\). Since the homogeneous
ideal of the irreducible cubic hypersurface \(S\) is generated by its cubic
equation \(F\), no nonzero quadratic can vanish on \(S\); the polar would
therefore vanish identically on \(\PP^{n+1}\). After a linear change of
coordinates this would give \(F_{x_0}=0\), making \(S\) a cone and therefore
singular.

The sets \(U\) and \(Y\) are nonempty Zariski open subsets of the irreducible
variety \(S\), while the loci just excluded are proper closed subsets.
Therefore we may choose
\[
x\in U\cap Y\setminus
\bigl((S\cap\mathbb T_yS)\cup P_y\cup\{y\}\bigr).
\]
Then \(\ell=\langle x,y\rangle\) is tangent to \(S\) at neither endpoint.
Moreover, \(\ell\not\subset S\): otherwise its tangent direction at \(y\)
would lie in \(T_yS\), so \(x\in\mathbb T_yS\), contrary to the choice of
\(x\). B\'ezout's theorem gives total intersection multiplicity three.
Nontangency at \(x\) and \(y\) makes both endpoint multiplicities equal to one,
so the intersection is reduced, the residual point \(\rho\) is distinct from them,
and
\[
S\cdot\ell=x+\rho+y.
\]
Let \(L_1,\ldots,L_n\) be the lines through \(x\)
supplied by \eqref{eq:kz-spanning-at-x}. None is contained in \(H_0\), since
\(x\notin H_0\), and none contains \(y\), since otherwise
\(\ell=\langle x,y\rangle=L_i\subset S\).

For each \(i\), set \(\Pi_i=\langle L_i,y\rangle\) and write
\[
\Pi_i\cap S=L_i+\Gamma_i,
\]
where \(\Gamma_i\) is the residual conic through \(y\). Since
\(\ell\subset\Pi_i\) but \(\ell\not\subset S\), the plane \(\Pi_i\) cannot
be contained in \(S\). By the local projection argument above, the third-point
involution \(\iota_\rho\) is biholomorphic near \(x\) and \(y\). Hence
\(\iota_\rho\) maps the germ of \(L_i\) at \(x\)
biholomorphically to the germ of \(\Gamma_i\) at \(y\), and
\[
T_y\Gamma_i=d\iota_\rho(T_xL_i).
\]

To check regularity at \(y\), write
\(F|_{\Pi_i}=\ell_iq_i\), with \(L_i=\{\ell_i=0\}\) and
\(\Gamma_i=\{q_i=0\}\). Since \(\ell_i(y)\ne0\),
\[
d(F|_{\Pi_i})_y=\ell_i(y)dq_i(y).
\]
Thus \(dq_i(y)=0\) would give \(d(F|_{\Pi_i})_y=0\), so
\(T_y\Pi_i\subset T_yS\), equivalently
\(\Pi_i\subset\mathbb T_yS\). This contradicts the transversality at \(y\) of
\(\ell\subset\Pi_i\). Hence \(\Gamma_i\) is regular at \(y\).

By \cref{prop:higher-one-line}, the vector field \(W_{L_i}(y)\) spans
\(T_y\Gamma_i\).  Hence
\[
\begin{aligned}
\operatorname{Span}
(W_{L_1}(y),\ldots,W_{L_n}(y))
&=
d\iota_\rho
\left(
\operatorname{Span}
(T_xL_1,\ldots,T_xL_n)
\right)
\\
&=
T_yS
=
T_yY .
\end{aligned}
\]
This proves the lemma.
\end{proof}

\subsection{Finite extraction}

For an ordered \(n\)-tuple
\(\mathbf L=(L_1,\ldots,L_n)\), let
\[
\Omega_{\mathbf L}
=
\left\{
y\in Y:
W_{L_1}(y)\wedge\cdots\wedge W_{L_n}(y)\neq0
\right\}.
\]
The wedge in the definition is an algebraic section of \(\det T_Y\), so every
\(\Omega_{\mathbf L}\) is Zariski open. These sets cover \(Y\) by
\cref{lem:pointwise-lines}. A Noetherian space is quasi-compact; hence this
open cover has a finite subcover. Hence there are an integer \(N_0\geqslant1\)
and tuples \(\mathbf L^{(1)},\ldots,\mathbf L^{(N_0)}\) such that
\[
Y=\Omega_{\mathbf L^{(1)}}\cup\cdots\cup\Omega_{\mathbf L^{(N_0)}}.
\]
The sprays associated with all lines occurring in these tuples form a finite
dominating family.

\begin{proof}[Proof of \cref{thm:higher-source}]
By the composition lemma \cite[Lemma~2.1]{ForstnericSub}, this family combines
into a single dominating spray. Hence \(Y\) is holomorphically elliptic.
\end{proof}

\section{Affine quadratic complements}\label{sec:quadratics}

We next treat reducible cubics. A linear factor places their complements in an
affine chart, where the remaining equation has degree at most two. Let \(q\)
be a nonzero polynomial in \(z_1,\ldots,z_n\) of degree at most two, and put
\(A_q=\C^n\setminus\{q=0\}\). Hanysz treated smooth affine plane conics
\cite[Appendix~B]{HanyszCubicMaps}, and Kusakabe proved that projective
quadric complements are Oka
\cite[Corollary~4.9(1)]{KusakabeLocalization}. The classification also
requires degenerate and nonreduced quadratic equations.

\begin{theorem}\label{thm:affine-quadratic}
Let \(n\geqslant1\).  The manifold \(A_q\) is elliptic, hence Oka,
unless an affine change of coordinates and multiplication of \(q\) by
a nonzero constant transform it into
\[
                            z_1(z_1-1).
\]
In the exceptional case
\[
               A_q\cong(\C\setminus\{0,1\})\times\C^{n-1}
\]
and \(A_q\) is not Oka.
\end{theorem}

\begin{proof}
Write
\[
                         q=q_2+\ell+a,
\]
where \(q_2\) is homogeneous quadratic, \(\ell\) is linear, and
\(a\in\C\).  By \(\rank q_2\) we mean the rank of the symmetric
bilinear form associated to \(q_2\).

Set \(k_q:=\rank q_2\), and assume first that \(k_q\geqslant2\). After a linear change of
coordinates,
\[
                         q_2=z_1^2+\cdots+z_{k_q}^2.
\]
Let \(e_1,\ldots,e_n\) be the standard coordinate vectors.  The vectors
\begin{align*}
 v_1&=e_1+ie_2, & v_2&=e_1-ie_2,\\
 v_j&=e_1+ie_j &&(3\leqslant j\leqslant k_q),\\
 v_j&=e_j       &&(k_q<j\leqslant n)
\end{align*}
form a basis of \(\C^n\), and every one is \(q_2\)-isotropic:
\(q_2(v_j)=0\).

For these isotropic directions, we apply the standard overshear construction. In Varolin's
coordinate-free terminology, if \(V\) is a \(\C\)-complete holomorphic vector
field and \(f\in\ker V^2\), then \(fV\) is a complete overshear vector field;
see \cite[\S3]{VarolinShears}. The explicit formula for its flow was given by
Andrist and Kutzschebauch
\cite[Lemma~3.3]{AndristKutzschebauchFibred}. We spell out the specialization
to \(V=D_v\) and \(f=q\), since it shows directly that the resulting flow
preserves \(A_q\). The point of the present argument is that the
\(q_2\)-isotropic vectors chosen above form a basis of \(\C^n\); the
corresponding overshears can therefore be combined into a dominating spray on
\(A_q\).

Fix a \(q_2\)-isotropic vector \(v\), let \(D_v\) denote differentiation in
the constant direction \(v\), and set \(b_v=D_vq\). Since \(q\) has degree at
most two, Taylor's formula gives
\begin{equation}\label{eq:q-affine-line}
 q(z+\sigma v)
   =q(z)+\sigma D_vq(z)+\sigma^2q_2(v)
   =q(z)+\sigma b_v(z).
\end{equation}
Thus \(q\) varies affinely along every line parallel to \(v\). Applying the
same argument to the affine-linear function \(b_v\), we obtain
\[
 b_v(z+\sigma v)
   =b_v(z)+\sigma D_v^2q(z)
   =b_v(z),
\]
because \(D_v^2q=2q_2(v)=0\). Hence \(b_v\) is constant along each such
line.

Consider the holomorphic vector field \(V_v=qD_v\). Its integral curve
through \(z\) has the form \(z+\sigma(t)v\), where
\cref{eq:q-affine-line} reduces the flow equation to
\[
 \sigma'(t)=q(z)+b_v(z)\sigma(t),\qquad \sigma(0)=0.
\]
Writing
\[
 \chi(w)=
 \begin{cases}
  (e^w-1)/w,&w\ne0,\\
  1,&w=0,
 \end{cases}
\]
the solution is \(\sigma(t)=t\chi\bigl(tb_v(z)\bigr)q(z)\). The function
\(\chi\) is entire; its value at the origin incorporates the case
\(b_v(z)=0\). Consequently, \(V_v\) is complete and its flow is
\begin{equation}\label{eq:quadratic-flow}
 \Phi_v^t(z)
   =z+t\chi\bigl(t b_v(z)\bigr)q(z)v,
 \qquad (z,t)\in\C^n\times\C.
\end{equation}

Substituting \cref{eq:quadratic-flow} into
\cref{eq:q-affine-line}, we obtain
\begin{equation}\label{eq:q-transformation}
 \begin{aligned}
 q\bigl(\Phi_v^t(z)\bigr)
   &=q(z)+t\chi\bigl(tb_v(z)\bigr)q(z)b_v(z)\\
   &=\Bigl[1+tb_v(z)\chi\bigl(tb_v(z)\bigr)\Bigr]q(z)\\
   &=e^{tb_v(z)}q(z),
 \end{aligned}
\end{equation}
Here the last equality follows from \(1+w\chi(w)=e^w\). Since the
exponential factor never vanishes, \cref{eq:q-transformation} shows that
the flow preserves \(A_q\).

Compose the flows belonging to the basis above:
\[
 s(z;t_1,\ldots,t_n)
   =\Phi_{v_n}^{t_n}\circ\cdots\circ\Phi_{v_1}^{t_1}(z).
\]
Each flow preserves \(A_q\), and \(\Phi_v^0=\id\); hence their composition is
a holomorphic spray \(A_q\times\C^n\to A_q\) fixing the zero section. Its differential
in the parameter variables at the zero section has columns
\(q(z)v_1,\ldots,q(z)v_n\).  Since \(q(z)\ne0\), these columns form a
basis of \(T_zA_q\).  Thus \(s\) is dominating and \(A_q\) is elliptic.

It remains to consider \(k_q\leqslant1\); repeated factors are included
in this analysis. Rank zero is the affine-linear case. In rank one, completing
the square will reduce the complement either to a product with \(\C^*\) or to
the exceptional one-variable polynomial. If \(q_2=0\), then
\(q\) is constant or affine linear.  Its complement is respectively
\(\C^n\) or \(\C^*\times\C^{n-1}\).  These manifolds are elliptic:
the additive and multiplicative complex groups have global dominating
sprays, and ellipticity is preserved by products.

Suppose \(k_q=1\).  Completing the square and making a linear
change in the remaining variables gives
\[
             q=z_1^2+b_2z_2+\cdots+b_nz_n+a.
\]
If some \(b_j\ne0\), we may assume after rescaling that \(b_2=1\).
The triangular polynomial automorphism
\[
 (z_1,z_2,z_3,\ldots,z_n)
       \longmapsto (z_1,q,z_3,\ldots,z_n)
\]
has inverse obtained by solving
\(q=z_1^2+z_2+\sum_{j=3}^n b_jz_j+a\) for \(z_2\); hence it identifies
\(A_q\) with \(\C^{n-1}\times\C^*\).
If every \(b_j=0\), then
\(q=z_1^2+a\).  For \(a=0\), the complement is again
\(\C^*\times\C^{n-1}\).  For \(a\ne0\), the polynomial has two
distinct roots, and an affine rescaling gives \(q=z_1(z_1-1)\).

Finally, \(\C\setminus\{0,1\}\) is not Oka by the classical
classification of Oka Riemann surfaces; see
\cite[Section~2]{Hanysz}.  It is a holomorphic retract of
\((\C\setminus\{0,1\})\times\C^{n-1}\), so the latter cannot be Oka.
\end{proof}

\begin{remark}
The proof shows that every nonexceptional \(A_q\) is holomorphically flexible
in the sense of \cite[Definition~6.4]{ArzhantsevFlexible}: at each point, the
values of \(\C\)-complete holomorphic vector fields span its tangent space.
For \(k_q\geqslant2\), the fields
\(qD_{v_1},\ldots,qD_{v_n}\) do so; the remaining nonexceptional cases are
\(\C^n\) and \(\C^*\times\C^{n-1}\). This is holomorphic rather than
algebraic flexibility: the overshear fields have algebraic coefficients, but
their flows are generally nonalgebraic and the fields need not be locally
nilpotent.
\end{remark}

\begin{corollary}\label{cor:reducible}
Let \(n\geqslant2\), and let \(D\subset\PP^n\) be a cubic hypersurface
whose defining polynomial is reducible; repeated factors are allowed.
Then \(\PP^n\setminus D\) is holomorphically elliptic unless
\(D_{\mathrm{red}}\) consists of three distinct hyperplanes containing a
common linear subspace of codimension two. In the exceptional case it is not
Oka.
\end{corollary}

\begin{proof}
The cubic equation has a linear factor, say \(F=LQ\), where \(L\) is
linear and \(Q\) has degree two; repetitions are allowed.  The
complement lies in the affine chart \(\{L\ne0\}\cong\C^n\).  Setting
\(L=1\) turns \(Q\) into a nonzero polynomial \(q\) of degree at most
two and gives
\[
                   \PP^n\setminus D\cong A_q.
\]
Apply \cref{thm:affine-quadratic}.

To translate the exceptional normal form projectively, choose coordinates
with \(L=Z_0\). Homogenizing \(q=z_1(z_1-1)\) gives
\[
F=Z_0Z_1(Z_1-Z_0).
\]
Thus \(D_{\mathrm{red}}\) consists of three distinct hyperplanes. Their common
intersection is given by \(Z_0=Z_1=0\), a linear space isomorphic to
\(\PP^{n-2}\). Conversely, after choosing one of any three
such hyperplanes as infinity, the other two become parallel affine
hyperplanes, whose product is affinely equivalent to \(z_1(z_1-1)\).
This agrees with the exceptional case of Hanysz's
hyperplane-arrangement theorem \cite[Theorem~3.1]{Hanysz}.
\end{proof}

\section{Projection from a singular point}
\label{sec:singular-projection}

For integral singular cubics, projection from a singular point lowers the
dimension and yields a recursive reduction. The complement is covered by two
Zariski-open pieces: a
\(\C^*\)-bundle over a quadric complement and a graph complement on a
cubic-root cover. Kusakabe's localization theorem then promotes these local
Oka descriptions to the whole complement.

\begin{lemma}\label{lem:singular-reduction}
Let \(n\geqslant2\), and let
\[
 F(x_0,x)=x_0q(x)+c(x)
\]
be a nonzero homogeneous cubic on \(\C^{n+1}\), where
\(x=(x_1,\ldots,x_n)\), \(q\) is homogeneous quadratic, and
\(0\ne c\) is homogeneous cubic.  If
\[
                    \PP^{n-1}\setminus\{c=0\}
\]
is Oka, then \(\PP^n\setminus\{F=0\}\) is Oka.
\end{lemma}

\begin{proof}
Let \(P=\{x_0=0\}\cong\PP^{n-1}\) be the hyperplane of directions from
\(p=[1:0:\cdots:0]\), and put \(\mathcal L=\OO_P(1)\). Let
\[
 \pi\colon\PP^n\setminus\{p\}\longrightarrow P,
 \qquad [x_0:x]\longmapsto[x],
\]
be the projection from \(p\). It identifies \(\PP^n\setminus\{p\}\) with
the total space of the line bundle \(\mathcal L\to P\). Explicitly,
\(b=[x]\in P\) represents the line
\(\ell=\C x\), the fibre
\(\mathcal L_b\) is \(\ell^*\), and \([x_0:x]\) corresponds to the functional
\(\varphi\in\ell^*\) satisfying \(\varphi(x)=x_0\). This description is
independent of the chosen representative \(x\).

In a local frame of \(\mathcal L\), denote the fibre coordinate by \(t\).  The
removed hypersurface has equation
\begin{equation}\label{eq:singular-fibre-equation}
                         tq+c=0.
\end{equation}
Here \(q\) and \(c\) are the local coefficient functions of sections of
\(\OO_P(2)\) and \(\OO_P(3)\). Consequently, \(-c/q\), where defined,
transforms as the coefficient of a section of \(\mathcal L=\OO_P(1)\):
under \(x\mapsto\lambda x\), the quotient scales by
\(\lambda^{3-2}=\lambda\).

For \(U=\PP^n\setminus\{F=0\}\), set
\[
P_q=P\setminus\{q=0\},\quad P_c=P\setminus\{c=0\},\qquad
U^{(q)}=U\cap\pi^{-1}(P_q),\qquad
U^{(c)}=U\cap\pi^{-1}(P_c).
\]
Since \(p\) lies on the removed cubic, \(\pi\) is algebraic on \(U\);
hence \(U^{(q)}\) and \(U^{(c)}\) are Zariski open in \(U\).

Assume first that \(q\not\equiv0\).  Over \(P_q\), the removed set in
\cref{eq:singular-fibre-equation} is the image of the holomorphic
section \(-c/q\) of \(\mathcal L\).  The fibrewise translation
\(t\mapsto t+c/q\) sends this section to the zero section and
identifies \(U^{(q)}\) with \(\mathcal L|_{P_q}\) minus its zero section.  Thus
\(U^{(q)}\to P_q\) is a holomorphic \(\C^*\)-bundle.  The base \(P_q\)
is the complement of a projective quadric and is Oka by
\cite[Corollary~4.9(1)]{KusakabeLocalization}.  Hence \(U^{(q)}\) is
Oka.

Over \(P_c\), take the cubic-root covering
\begin{equation}\label{eq:singular-cubic-cover}
 M_c=\{x\in\C^n:c(x)=1\}\longrightarrow P_c,
                         \qquad x\longmapsto[x].
\end{equation}
Since \(c\) is homogeneous of degree three, Euler's identity gives
\[
 \sum_jx_j\frac{\partial c}{\partial x_j}=3c.
\]
On \(M_c=\{c=1\}\), the right-hand side equals \(3\). Hence the
differential \(dc\) cannot vanish at any point of \(M_c\), and \(M_c\)
is smooth. Scalar multiplication by \(\mu_3\) acts
freely, with quotient \(P_c\); hence \(M_c\to P_c\) is an unramified
three-sheeted covering. Since \(P_c\) is Oka, so is \(M_c\).

The pullback of \(\mathcal L\) to \(M_c\) is trivial.  The normalized
representative \(x\), for which \(c(x)=1\), trivializes this pullback
by evaluation: it sends \(\varphi\in(\C x)^*\) to the scalar
\(t=\varphi(x)\).  In this trivialization the pullback of
\(U^{(c)}\) is
\begin{equation}\label{eq:singular-graph-complement}
 \mathcal W=\{(x,t)\in M_c\times\C:1+tq(x)\ne0\}.
\end{equation}
It is the complement of the affine part of the graph of the meromorphic
function
\[
m=-\frac1q:M_c\longrightarrow\PP^1.
\]
This function has Hanysz's required form \(f+1/g\), with \(f=0\) and
\(g=-q\).
At a zero of \(q\), the value of \(m\) is infinity and the entire
affine fibre belongs to \(\mathcal W\).  Hence \cref{thm:hanysz-graph} applies
and shows that \(\mathcal W\) is Oka.

The map \(\mathcal W\to U^{(c)}\), \((x,t)\mapsto[t:x]\), has deck action
\((x,t)\mapsto(\zeta x,\zeta t)\), \(\zeta\in\mu_3\). This action is
free: since \(x\ne0\), the equality \(\zeta x=x\) forces \(\zeta=1\).
Moreover, \(1+tq(x)\) is invariant because
\(q(\zeta x)=\zeta^2q(x)\). Hence
\(\mathcal W\to U^{(c)}=\mathcal W/\mu_3\) is an unramified finite covering. Since the Oka
property descends under such coverings, \(U^{(c)}\) is Oka.

The two Zariski-open subsets \(U^{(q)},U^{(c)}\subset U\) cover \(U\).
Indeed, if \(q(x)=c(x)=0\), then
\(tq(x)+c(x)=0\) for every point of the projection fibre.  The whole
fibre belongs to the removed cubic and contributes no point to \(U\).
Kusakabe's localization theorem, \cref{thm:kusakabe-localization},
now gives that \(U\) is Oka.

If \(q\equiv0\), then \(P_q\) and \(U^{(q)}\) are empty, while
\(U=U^{(c)}\).  Equation \cref{eq:singular-fibre-equation} is then
independent of the fibre coordinate, and \(U\) is the total space of
\(\mathcal L|_{P_c}\to P_c\). It is Oka by the Oka principle for holomorphic
fibre bundles.
\end{proof}

\begin{theorem}\label{thm:integral-singular}
Let \(D\subset\PP^n\), \(n\geqslant2\), be an integral singular cubic
hypersurface.  Then \(\PP^n\setminus D\) is Oka and hence holomorphically
elliptic.
\end{theorem}

\begin{proof}
We argue by induction on \(n\). The smooth input used below is
\cref{thm:smooth-complement}; its proof depends only on
\cref{thm:surface-source,thm:higher-source} and is independent of the present
induction. For \(n=2\), this is
Kusakabe's theorem for integral singular plane cubics
\cite[Corollary~4.9(2)]{KusakabeLocalization}.

Assume \(n\geqslant3\) and choose \(p\in\Sing D\).  Hyperplanes avoiding
\(p\) form a nonempty open set in the dual projective space.  By Bertini
irreducibility \cite[Th\'eor\`eme~6.3(4)]{Jouanolou}, a general such
hyperplane \(H\) has irreducible section \(C=D\cap H\).  We may also choose
\(H\) transverse to the dense smooth locus of \(D\), so \(C\) is
generically reduced.  As an effective Cartier divisor on the smooth variety
\(H\) with irreducible support, \(C\) therefore has multiplicity one and is
integral.  Thus \(C\subset H\cong\PP^{n-1}\) is an integral cubic. Choose
coordinates with
\[
p=[1:0:\cdots:0],
\qquad
H=\{x_0=0\}.
\]
Write the defining cubic polynomial of \(D\) as
\[
F=ax_0^3+x_0^2\ell(x)+x_0q(x)+c(x).
\]
Evaluating \(F\) at \(p\) gives \(a=0\). For \(1\leqslant j\leqslant n\),
the value \(F_{x_j}(p)\) is precisely the coefficient of \(x_j\) in
\(\ell\). Since \(p\) is singular, all these coefficients vanish. Hence
\(\ell=0\), so \(F=x_0q+c\), and \(c=F|_H\) defines \(C\). If \(C\) is
smooth, then \(H\setminus C\) is Oka by
\cref{thm:smooth-complement}. If \(C\) is singular, then
\(H\setminus C\) is Oka by the induction hypothesis. Therefore
\[
\PP^{n-1}\setminus C
\]
is Oka in either case.

The singular-point reduction \cref{lem:singular-reduction} now gives that
\(\PP^n\setminus D\) is Oka. Since this complement is affine and Stein,
Gromov's converse theorem makes it holomorphically elliptic.
\end{proof}

\section{Classification and applications to mapping spaces}
\label{sec:classification}

\subsection{Proof of the classification theorem}
\label{subsec:classification-proof}

We combine the smooth, reducible, and integral singular cases proved in the
preceding sections.

\begin{theorem}\label{thm:smooth-complement}
The complement of every smooth cubic hypersurface in \(\PP^n\),
\(n\geqslant2\), is holomorphically elliptic.
\end{theorem}

\begin{proof}
Write \(C=\{F=0\}\) and consider the cyclic cubic cover
\[
S_F=\{F(x)=w^3\}\subset\PP^{n+1},
\qquad B_F=S_F\cap\{w=0\}.
\]
Let \(G(x,w)=F(x)-w^3\). Away from \(w=0\), the derivative
\(G_w=-3w^2\) does not vanish. On \(w=0\), a singular point of \(S_F\)
would satisfy \(F=0\) and \(dF=0\), hence would give a singular point of
\(C\). Thus \(S_F\) is smooth. Moreover, \(B_F\cong C\), while setting \(w=1\) identifies
\(S_F\setminus B_F\) with \(\{F=1\}\subset\C^{n+1}\).
This affine hypersurface is elliptic by \cref{thm:surface-source} for \(n=2\)
and by \cref{thm:higher-source} for \(n\geqslant3\).

The radial projection
\[
\{F=1\}\longrightarrow\PP^n\setminus C,\qquad x\longmapsto[x],
\]
is an unramified three-sheeted covering. For each \([x]\) in the complement,
the equation \(\lambda^3F(x)=1\) has three solutions, and they form a free
\(\mu_3\)-orbit. Hence
\(\PP^n\setminus C\) is Oka and, being Stein, is elliptic by
\cref{sec:sprays}.
\end{proof}

\begin{proof}[Proof of \cref{thm:main}]
Every cubic is either smooth, irreducible and singular, or reducible; a
nonreduced cubic belongs to the last case. Thus
\cref{thm:smooth-complement,thm:integral-singular,cor:reducible} show that every
nonexceptional complement is Oka. Since it is affine and Stein, it is elliptic by
\cite[Proposition~5.6.15, p.~230]{ForstnericBook}.

In the exceptional case,
\(\PP^n\setminus D\cong(\C\setminus\{0,1\})\times\C^{n-1}\).
Its non-Oka first factor is a holomorphic retract, so the complement is
neither Oka nor elliptic.
\end{proof}

\begin{remark}\label{rem:algebraic-ellipticity}
The method of this paper proves holomorphic ellipticity but does not establish
algebraic ellipticity. If \(D\subset\PP^2\) is a smooth cubic, then
\(\PP^2\setminus D\) is not algebraically elliptic. Indeed, if
\(s:E\to\PP^2\setminus D\) is a dominating algebraic spray, then the maps
\(t\mapsto s(tv)\), with \(v\) ranging over the open subset on which the fibre
differential at the zero section is nonzero, form a dominant family of
nonconstant algebraic maps \(\A^1\to\PP^2\setminus D\). Thus the complement is
separably \(\A^1\)-uniruled in characteristic zero. By
\cite[Proposition~2.5 and Corollary~2.8]{ChenZhuLog}, this contradicts
\(K_{\PP^2}+D\cong\OO_{\PP^2}\). For a smooth cubic hypersurface
\(D\subset\PP^n\) with
\(n\geqslant3\), it is not known whether \(\PP^n\setminus D\) is algebraically
elliptic.
\end{remark}

\subsection{Degree-three elliptic functions}
\label{subsec:elliptic-functions}

We use the following elementary consequence of the Oka--Grauert principle.

\begin{lemma}\label{lem:principal-bundle-elliptic}
Let \(\pi:X\to B\) be a holomorphic principal \(G\)-bundle, where \(G\) is a
complex Lie group. If \(B\) is Stein and holomorphically elliptic, then
\(X\) is holomorphically elliptic.
\end{lemma}

\begin{proof}
Let \(s:V\to B\) be a dominating spray on a holomorphic vector bundle
\(p:V\to B\). The homotopy
\[
H_t(v)=s(tv),\qquad 0\leqslant t\leqslant1,
\]
joins \(p\) to \(s\), so the principal bundles \(p^*X\) and \(s^*X\) over
\(V\) are topologically isomorphic. Since the total space \(V\) of a
holomorphic vector bundle over the Stein manifold \(B\) is Stein, the
Oka--Grauert principle for principal bundles
\cite{GrauertFibrations} (see also
\cite[Theorem~8.2.1(i), p.~356]{ForstnericBook}) gives a holomorphic isomorphism
\[
\Phi:p^*X\xrightarrow{\sim}s^*X.
\]
On the zero section, \(\Phi\) restricts to a holomorphic bundle
automorphism \(a:X\to X\) over \(B\). Since \(\Phi\) is a principal-bundle
isomorphism, so is \(a\). Precomposing \(\Phi\) with the
automorphism
\[
p^*X\longrightarrow p^*X,\qquad (v,x)\longmapsto(v,a^{-1}(x)),
\]
we may assume that \(\Phi\) restricts to the identity on the zero section.

Using the natural identification \(p^*X\cong\pi^*V\), write
\[
\Phi(v,x)=(v,\widetilde s(x,v)).
\]
Then \(\widetilde s(x,0)=x\) and
\[
\pi\bigl(\widetilde s(x,v)\bigr)=s(v).
\]
Differentiating \(\pi\circ\widetilde s=s\) in the \(V\)-direction along the
zero section shows that \(d\widetilde s\) surjects onto
\(TX/\ker d\pi\), since \(s\) is dominating. The remaining vertical
directions are supplied by the principal action: if \(\mathfrak g\) is the
Lie algebra of \(G\), then
\[
S(x,v,\xi)=\widetilde s(x,v)\cdot\exp(\xi)
\]
is a dominating spray on
\((\pi^*V)\oplus(X\times\mathfrak g)\).
\end{proof}

\begin{remark}
The Stein and principal-bundle hypotheses are used in the lifting step.
The corresponding question for an arbitrary elliptic Serre fibration over an
elliptic base is open; see
\cite[Problem~6.4.19]{ForstnericBook}.
\end{remark}

\begin{proof}[Proof of \cref{thm:elliptic-functions}]
Put \(\mathcal R=\mathcal R_3(E,\PP^1)\). Bowman associates to \(E\) a smooth
plane cubic \(C_E\) and proves
\[
\mathcal R\text{ is Oka}
\quad\Longleftrightarrow\quad
\PP^2\setminus C_E\text{ is Oka}
\]
\cite[Corollary~30]{Bowman}. The right-hand side holds by
\cref{thm:smooth-complement}; hence \(\mathcal R\) is Oka.

Precomposition by translations of \(E\) is free, and Bowman proves that the
quotient map is a holomorphic principal bundle
\[
E\longrightarrow\mathcal R\longrightarrow B:=\mathcal R/E
\]
whose base \(B\) is Stein \cite[Theorems~39 and~40]{Bowman}. Since the fibre
\(E\) and the total space \(\mathcal R\) are Oka,
\cite[Theorem~5.6.5]{ForstnericBook} shows that \(B\) is Oka. A Stein Oka
manifold is elliptic \cite[Proposition~5.6.15]{ForstnericBook}; thus \(B\) is
elliptic. Applying \cref{lem:principal-bundle-elliptic} to the displayed
principal \(E\)-bundle proves that \(\mathcal R\) is elliptic.
\end{proof}

\begin{remark}[Geometric origin of the plane cubic]
The cubic \(C_E\) is intrinsic to the degree-three mapping space. Building on
Namba's lemma \cite[Lemma~1.4.1]{Namba}, Bowman obtains a finite unbranched
cover of the postcomposition quotient by \((\PP^2\setminus C_E)\times E\); the cover arises
from the multiplication-by-three isogeny of \(E\)
\cite[Proposition~28]{Bowman}. This explains geometrically why a plane-cubic
complement appears in the criterion above.
\end{remark}

\subsection{Cubic rational self-maps of the sphere}
\label{subsec:cubic-rational-maps}

The proof of \cref{thm:cubic-rational-maps} first identifies the
postcomposition quotient with the base-point-free locus in
\(\operatorname{Gr}(2,4)\), and then uses the B\'ezout matrix and projective
inversion to turn that locus into the complement of a cubic hypersurface in
\(\PP^4\). Put
\[
V=H^0(\PP^1,\OO_{\PP^1}(3)),
\]
and let
\[
\mathcal B
=
\{W\in\operatorname{Gr}(2,V):W\text{ is base-point-free on }\PP^1\}.
\]
For \(f=[p:q]\in R_3\), write \(W_f=\Span(p,q)\in\mathcal B\). The key
geometric identification is the following.

\begin{theorem}
\label{thm:cubic-rational-quotient}
There is an algebraic principal \(\operatorname{PGL}_2(\C)\)-bundle
\[
\operatorname{PGL}_2(\C)
\longrightarrow
R_3
\longrightarrow
\PP^4\setminus D,
\]
where
\[
D=
\left\{
XZW-XV^2-Y^2W+2YVZ-Z^3=0
\right\}
\subset\PP^4.
\]
Moreover, \(D\) is irreducible.
\end{theorem}

\begin{proof}
We proceed in three steps.

\medskip
\noindent\textbf{Step 1: the postcomposition quotient.}
\smallskip

The map
\[
\rho:R_3\longrightarrow\mathcal B,
\qquad
[p:q]\longmapsto\Span(p,q),
\]
is surjective. For \(W\in\mathcal B\), every ordered basis \((p,q)\) has no
common zero and therefore defines a point \([p:q]\in R_3\). Two ordered
bases define the same point of \(R_3\) precisely when they differ by a
common scalar. Thus
\(\operatorname{PGL}_2(\C)\) acts simply transitively on each fibre of
\(\rho\) by postcomposition. Equivalently, \(\rho\) is the projective frame
bundle of the tautological rank-two bundle restricted to \(\mathcal B\):
\begin{equation}\label{eq:R3-principal-bundle}
\operatorname{PGL}_2(\C)
\longrightarrow
R_3
\xrightarrow{\ \rho\ }
\mathcal B.
\end{equation}
See also \cite[\S3]{HanyszCubicMaps}.

\medskip
\noindent\textbf{Step 2: the B\'ezout matrix.}
\smallskip

Write
\[
\begin{aligned}
p&=a_0z_0^3+a_1z_0^2z_1+a_2z_0z_1^2+a_3z_1^3,\\
q&=b_0z_0^3+b_1z_0^2z_1+b_2z_0z_1^2+b_3z_1^3,
\end{aligned}
\]
and let
\[
P_{ij}=a_i b_j-a_j b_i,
\qquad 0\leqslant i<j\leqslant3,
\]
be the Pl\"ucker coordinates of \(W=\Span(p,q)\). The Pl\"ucker embedding
identifies \(\operatorname{Gr}(2,V)\) with the quadric
\begin{equation}\label{eq:plucker-R3}
P_{01}P_{23}-P_{02}P_{13}+P_{03}P_{12}=0
\end{equation}
in \(\PP^5\).

The classical B\'ezout formula associates to \(W\) the symmetric matrix
\begin{equation}\label{eq:bezout-matrix-cubics}
A(W)=
\begin{pmatrix}
P_{01}&P_{02}&P_{03}\\
P_{02}&P_{03}+P_{12}&P_{13}\\
P_{03}&P_{13}&P_{23}
\end{pmatrix}
\end{equation}
and gives
\[
\det A(W)=\pm\operatorname{Res}(p,q);
\]
see \cite[Example~4.1 and Proposition~4.2]{EisenbudSchreyerChow}. Hence
\[
W\in\mathcal B
\quad\Longleftrightarrow\quad
\det A(W)\neq0.
\]

Set
\[
A=
\begin{pmatrix}
x&y&z\\
y&u&v\\
z&v&w
\end{pmatrix},
\qquad
(x,y,z,u,v,w)
=
(P_{01},P_{02},P_{03},P_{03}+P_{12},P_{13},P_{23}).
\]
This linear change of coordinates identifies
\(\PP^5\) with \(\PP(\operatorname{Sym}_3(\C))\), where
\(\operatorname{Sym}_3(\C)\) denotes the six-dimensional vector space of
symmetric \(3\times3\) complex matrices.

Since \(P_{12}=u-z\), the
Pl\"ucker equation becomes
\begin{equation}\label{eq:quadric-symmetric-matrices}
Q(A):=xw-yv+zu-z^2=0.
\end{equation}
Therefore
\[
\mathcal B
\cong
\{[A]\in\PP(\operatorname{Sym}_3(\C)):Q(A)=0,\ \det A\neq0\}.
\]

\medskip
\noindent\textbf{Step 3: projective matrix inversion.}
\smallskip

Let \(\operatorname{adj}(A)\) denote the classical adjugate matrix,
characterized by \(A\operatorname{adj}(A)=\det(A)I\). On the open subset
\[
\{\det A\ne0\}\subset\PP(\operatorname{Sym}_3(\C)),
\]
the map \(\iota:[A]\mapsto[\operatorname{adj}(A)]\) is a biregular
involution, since
\(\operatorname{adj}(\operatorname{adj}(A))=(\det A)A\).
Put \(B=\operatorname{adj}(A)\). Since the entries of \(B\) read as
\[
B_{22}=xw-z^2,
\qquad
B_{13}=yv-zu,
\]
we have
\[
Q(A)=B_{22}-B_{13}.
\]
Thus \(\iota\) identifies \(\mathcal B\) with
\[
H\setminus\{\det B=0\},
\qquad
H:=\{B_{22}=B_{13}\}\cong\PP^4.
\]
Writing
\[
B=
\begin{pmatrix}
X&Y&Z\\
Y&Z&V\\
Z&V&W
\end{pmatrix}
\]
on \(H\), we obtain
\[
\det B=XZW-XV^2-Y^2W+2YVZ-Z^3.
\]
Hence \(\mathcal B\cong\PP^4\setminus D\).

It remains to prove that \(D\) is irreducible. Set
\[
R=\C[Y,Z,V,W],\qquad a=ZW-V^2,\qquad
b=-Y^2W+2YVZ-Z^3.
\]
Then \(F=aX+b\in R[X]\). The polynomial \(a\) is irreducible in the UFD
\(R\) (view it as a primitive linear polynomial in \(W\) over
\(\C[Y,Z,V]\)), and \(a\nmid b\): at \(Z=1\) and \(V=W=0\), one has \(a=0\) but
\(b=-1\). Hence \(\gcd(a,b)=1\), so \(F\) is primitive. It is a nonconstant
linear polynomial over \(\operatorname{Frac}(R)\), and is therefore
irreducible there. Gauss' lemma now gives irreducibility in \(R[X]\).
\end{proof}

\begin{proof}[Proof of \cref{thm:cubic-rational-maps}]
By \cref{thm:cubic-rational-quotient}, \(R_3\) is a holomorphic principal
\(\operatorname{PGL}_2(\C)\)-bundle over \(\PP^4\setminus D\), where
\(D\) is an irreducible cubic hypersurface. Hence \(D\) is not the
exceptional union of three hyperplanes in \cref{thm:main}, so
\(\PP^4\setminus D\) is elliptic and, being affine, Stein. Applying
\cref{lem:principal-bundle-elliptic} to the displayed principal bundle shows
that \(R_3\) is holomorphically elliptic.
\end{proof}

\begin{remark}[Scope of the B\'ezout reduction]
\label{rem:degree-three-special}
The quotient by base-point-free pencils and its B\'ezout description exist
in every degree. For \(d=2\), the quotient is \(\PP^2\) minus the conic of
pencils with a base point. For \(d\geqslant2\), the postcomposition quotient of \(R_d\) is the
base-point-free locus in
\[
\operatorname{Gr}\bigl(2,H^0(\PP^1,\OO_{\PP^1}(d))\bigr)
\cong\operatorname{Gr}(2,d+1),
\]
and the B\'ezout formula represents the resultant by the determinant of a
symmetric \(d\times d\) matrix linear in the Pl\"ucker coordinates
\cite[Example~4.1 and Proposition~4.2]{EisenbudSchreyerChow}.

For \(d=3\), two coincidences occur: \(\operatorname{Gr}(2,4)\subset\PP^5\)
is a quadric hypersurface, and projective inversion of the associated
symmetric \(3\times3\) B\'ezout matrix turns this single quadratic equation
into a linear one on the invertible locus. This identifies the quotient with
a cubic hypersurface complement in \(\PP^4\). For \(d\geqslant4\), the
Grassmannian has higher codimension, so there is no single Pl\"ucker equation
for inversion to linearize. Thus this particular one-equation argument does
not directly produce a hypersurface-complement model for the quotient of
\(R_d\). This is only a limitation of the present B\'ezout reduction; it does
not rule out a different model in higher degree.
\end{remark}

\medskip

\section*{Acknowledgements}

The author is grateful to Franc Forstneri\v c for introducing him to Oka
theory and for his careful reading and valuable comments on an earlier
version of the manuscript. He thanks Yuta Kusakabe for drawing his attention
to the problem of cubic complements. He also thanks Gaofeng Huang for helpful
suggestions that improved the manuscript.

\section*{AI use disclosure}

The author developed the mathematical ideas and overall proof strategy and
used OpenAI's GPT-5.6-Sol as an auxiliary tool for drafting, language editing,
and working through certain technical points. The author verified all
mathematical statements and arguments in the final manuscript and takes full
responsibility for its content.

\section*{Funding}

The author acknowledges partial support from the National Key R\&D
Program of China under Grants No.~2023YFA1010500 and
No.~2021YFA1003100, and from the National Natural Science Foundation
of China under Grants No.~12288201 and No.~12471081, as well as
support from the Xiaomi Young Talents Program.

\end{document}